\documentclass[10pt]{amsart}
\usepackage{graphicx}
\usepackage{mathrsfs}
\usepackage{amsfonts}
\usepackage{amssymb}
\usepackage{amsmath,amsthm}
\usepackage{hyperref}
\usepackage{color}
\def\opn#1#2{\def#1{\operatorname{#2}}} 
\opn\chara{char} \opn\length{\ell}
\allowdisplaybreaks
 \opn\projdim{proj\,dim} \opn\injdim{inj\,dim}
\opn\rank{rank} \opn\depth{depth} \opn\grade{grade}
\opn\height{height} \opn\embdim{emb\,dim} \opn\codim{codim}

\opn\Tr{Tr} \opn\bigrank{big\,rank}
\opn\superheight{superheight}\opn\lcm{lcm}
\opn\trdeg{tr\,deg}%
\opn\reg{reg} \opn\lreg{lreg}

\opn\Ker{Ker} \opn\Coker{Coker} \opn\Im{Im} \opn\Hom{Hom}
\opn\Tor{Tor} \opn\Ext{Ext} \opn\End{End} \opn\Aut{Aut} \opn\id{id}

\opn\nat{nat}
\opn\pff{pf}
\opn\Pf{Pf} \opn\GL{GL} \opn\SL{SL} \opn\mod{mod} \opn\ord{ord}

\def\Implies{\ifmmode\Longrightarrow \else
     \unskip${}\Longrightarrow{}$\ignorespaces\fi}
\def\implies{\ifmmode\Rightarrow \else
     \unskip${}\Rightarrow{}$\ignorespaces\fi}
\def\iff{\ifmmode\Longleftrightarrow \else
     \unskip${}\Longleftrightarrow{}$\ignorespaces\fi}

\let\:=\colon

\newtheorem{Theorem}{Theorem}[section]
\newtheorem{Lemma}[Theorem]{Lemma}

\newtheorem{Remark}[Theorem]{Remark}

\newtheorem{Example}[Theorem]{Example}

\theoremstyle{definition}

\let\epsilon=\varepsilon
\let\kappa=\varkappa
\opn\ini{in} \opn\inm{inm} \opn\Sym{Sym} \opn\diag{diag}
\opn\Ii{(i)} \opn\Iii{(ii)}

\title{A finite element method for Dirichlet boundary control problems governed by parabolic PDEs}
\author{Wei Gong $^\dag$}
\author{Michael Hinze $^*$}
\author{Zhaojie Zhou $^{\diamond}$}
\thanks{$^\dag$ LSEC, Institute of Computational Mathematics, Academy of Mathematics and Systems Science, Chinese Academy of Sciences, Beijing 100190, China.  Email: {\tt wgong@lsec.cc.ac.cn}}
\thanks{$^*$ Schwerpunkt Optimierung und Approximation, Universit\"{a}t Hamburg, Bundesstrasse 55, 20146, Hamburg, Germany. Email: {\tt michael.hinze@uni-hamburg.de}}
\thanks{$^{\diamond}$  School of Mathematics Sciences, Shandong Normal University, 250014, Ji'nan, China.
 Email: {\tt zzj534@amss.ac.cn} }

\date{\today}

\begin{document}
\maketitle

{\bf Abstract:}\hspace*{10pt} {
Finite element approximations of Dirichlet boundary control problems governed by parabolic PDEs on convex polygonal domains are studied in this paper. The existence of a unique solution to optimal control problems is guaranteed based on very weak solution of the state equation and $L^2(0,T;L^2(\Gamma))$ as control space. For the numerical discretization of the state equation we use standard piecewise linear and continuous finite elements for the space discretization of the state, while a dG(0) scheme is used for time discretization. The Dirichlet boundary control is realized through a space-time $L^2$-projection. We consider both piecewise linear, continuous finite element approximation and variational discretization for the controls and derive a priori $L^2$-error bounds for controls and states. We finally present numerical examples to support our theoretical findings.
}

{{\bf Keywords:}\hspace*{10pt}optimal control problem, parabolic equation, finite element method, a priori error estimate, Dirichlet boundary control}

\section{Introduction}
\setcounter{equation}{0}
In this paper we study the  following parabolic optimal control problem:
\begin{eqnarray}
\min\limits_{u\in U_{ad}}\ \ J(y,u)&={1\over 2}\|y-y_d\|^2_{L^2(0,T;L^2(\Omega))}+{\alpha\over 2}\|u\|^2_{L^2(0,T;L^2(\Gamma))}\label{OPT}
\end{eqnarray}
subject to
\begin{equation}\label{OPT_state}
\left\{ \begin{aligned}\frac{\partial y}{\partial t} -\Delta y=f \ \ &\mbox{in}\
\Omega_T, \\
 \ y=u  \ \ \ &\mbox{on}\ \Sigma,\\
 y(0)=y_0\ \ \ &\mbox{in}\ \Omega,
\end{aligned} \right.
\end{equation}
where $\Omega_T=\Omega\times (0,T]$, $\Sigma=\partial\Omega\times (0,T]$ with $\Omega$ denoting an open bounded, convex polygonal domain in $\mathbb{R}^2$ with boundary $\Gamma:=\partial\Omega$, and $\alpha>0$, $f\in L^2(0,T;L^2(\Omega))$, $y_0\in L^2(\Omega)$, $y_d\in L^2(0,T;L^2(\Omega))$ and $T>0$ are fixed. $U_{ad}$ is the admissible control set which is assumed to be of box type
\begin{eqnarray}
U_{ad}:=\big\{u\in L^2(0,T;L^2(\Gamma)):\ u_a\leq u(x,t)\leq u_b,\ \mbox{a.e.}\ \mbox{on}\ \Sigma \big\}\label{control_set},
\end{eqnarray}
with $u_a< u_b$ denoting constants.

Dirichlet boundary control is important in many practical applications such as the active boundary control of flows. If one is, e.g., interested in blowing and suction as control on part of the boundary, controls with low regularity should be admissible, which could have jumps and satisfy pointwise bounds. In the mathematical theory one has to use the concept of very weak solutions in this situation, see \cite{Berggren.M2004} for a more detailed discussion of this fact.

In the present work we consider a parabolic Dirichlet boundary control problem of tracking type, which may be regarded as prototype problem to study Dirichlet boundary control for time-dependent PDEs. For parabolic optimal boundary control problems of Dirichlet type, only few contributions can be found in the literature. Kunisch and Vexler  in \cite{Kunisch.K;Vexler.B2007} considered a semi-smooth Newton method for the numerical solution of parabolic Dirichlet boundary control problems. A Robin penalization method using Robin-type boundary conditions applied to parabolic Dirichlet boundary control problems is investigated in \cite{Belgacem.FB;Bernardi.C;Fekih.HE2011}. However, to the best of the authors' knowledge no error analysis is available for the finite element approximation of this kind of problems. With the present paper we intend to fill this gap and derive a priori error estimates for parabolic Dirichlet boundary control problems. Compared to the elliptic case, parabolic Dirichlet boundary control problems are more involved in both the definition of discrete schemes and the a priori error analysis, since the regularity of the involved state variable is low.

Finite element approximations of optimal control problems are important for the numerical treatment of optimal control problems related to practical applications, see e.g. \cite[Ch.4]{HPUU2009}. An overview on the numerical a priori and a posteriori analysis for elliptic control problems can be found in \cite[Ch.3]{HPUU2009} and \cite{Liu.WB;Yan.NN2008}, respectively. To the best of the authors knowledge the first contribution to parabolic optimal control problems is given in \cite{Winther.B1978}. The state of the art in the numerical a priori analysis of distributed parabolic optimal control problems can be found in \cite{Meidner.D;Vexler.B2008A,Meidner.D;Vexler.B2008B}. More recent contributions with higher order in time Galerkin schemes can be found in \cite{Apel.T;Flaig.TG2010,Meidner.D;Vexler.B2011,DHV2014}. For boundary control problems with parabolic equations we refer to \cite{Gong.W;Yan.NN2009}. Residual-based a posteriori error estimates are presented in \cite{Liu.WB;Ma.HP;Tang.T;Yan.NN2004} and \cite{Liu.WB;Yan.NN2003}. There is a long list of contributions to boundary control of elliptic PDEs, see e.g. \cite{Casas.E;Raymond.JP2006,Deckelnick.K;Gunther.A;Hinze.M2009SICON,French.DA;King.JT1991,Geveci.T1979,Gong.W;Yan.NN2011SICON,HinzeMatthes2009,May.S;Rannacher.R;Vexler.B2008,Vexler.B2007}. Further references can be found in \cite[Ch.3]{HPUU2009}.

In this paper we use the very weak solution concept for the state equation and $L^2(0,T;L^2(\Gamma))$ as control space to argue the existence of a unique solution to the optimal control problems (\ref{OPT})-(\ref{OPT_state}). For the numerical discretization of the optimal control problem we discretize the state using standard piecewise linear and continuous finite elements in space and dG(0) scheme in time. The Dirichlet boundary conditions are approximated based on the space-time $L^2$-projection. The control is discretized in space either by piecewise linear finite elements or implicitly through the discretization of the adjoint state, the so-called variational discretization (see \cite{Hinze.M2005}). For both cases we derive a priori error bounds for the state and control in the $L^2$-norm for problems posed on polygonal domains. As main result we obtain the error bound
\begin{eqnarray}
\|u-U_{hk}\|_{L^2(L^2(\Gamma))}+\|y-Y_{hk}\|_{L^2(L^2(\Omega))}\leq (h^{1\over 2}+k^{{1\over 4}})\label{main}
\end{eqnarray}
for the optimal solution $(y,u)$ of (\ref{OPT}), where $Y_{hk}$ and $U_{hk}$ denote the optimal discrete state and control. We present several numerical examples which support our theoretical findings.

The rest of our paper is organized as follows. In Section 2 we
present the analytical setting of the parabolic Dirichlet boundary control problem and argue the existence of a unique solution. In Section 3 we establish the fully discrete finite element approximation to the state equation and the corresponding stability results. Then we formulate the fully discrete approximation for parabolic Dirichlet boundary control problems. The a priori error
analysis for the finite element approximation and the variational discretization of the optimal control problems
posed on convex, polygonal or polyhedral domains is studied in Section
4. Furthermore, we present some numerical
experiments  in Section 5 to support our theoretical results.


\section{Optimal control problem}\setcounter{equation}{0}
Let $\Omega\subset \mathbb{R}^2$ be a polygonal convex domain. For $m\geq0$ and $1\leq s\leq\infty$, we adopt the standard notation $W^{m,s}(\Omega)$ for Sobolev spaces on $\Omega$ with norm $\|\cdot\|_{m,s,\Omega}$ and seminorm $|\cdot|_{m,s,\Omega}$, where $H^m(\Omega)=W^{m,2}(\Omega)$, $\|\cdot\|_{m,\Omega}=\|\cdot\|_{m,2,\Omega}$ and $|\cdot|_{m,\Omega}=|\cdot|_{m,2,\Omega}$ for $s=2$. Note that $H^0(\Omega)=L^2(\Omega)$ and $H_0^1(\Omega)=\{v\in H^1(\Omega);\ v=0\ \mbox{on}\ \partial\Omega\}$.
We denote by $L^r(0,T;W^{m,s}(\Omega))$ the Banach space of all $L^r$ integrable functions from $[0,T]$ into $W^{m,s}(\Omega)$ with norm $\|v\|_{L^r(0,T;W^{m,s}(\Omega))}=
\Big(\int_0^T\|v\|^r_{m,s,\Omega}dt\Big)^{{1\over r}}$ for $1\leq r<\infty$, and with the standard modification for $r=\infty$. For a Banach space $Y$, we use the abbreviations $L^2(Y)=L^2(0,T;Y)$, $H^s(Y)=H^s(0,T;Y)$, $s=[0,\infty)$, and $C(Y)=C([0,T];Y)$. We denote the $L^2$-inner products
on $L^2(\Omega)$, $L^2(\Omega_T)$ and $L^2(\Gamma)$ by $(\cdot,\cdot)$, $(\cdot,\cdot)_{\Omega_T}$ and $\langle\cdot,\cdot\rangle$, respectively. In addition $c$ and $C$ denote  generic positive constants.


Let
\begin{eqnarray*}
a(y,w)=\int_{\Omega}\nabla y\cdot\nabla w &\forall\ y,w\in
H^1(\Omega).
\end{eqnarray*}
The standard weak form for the parabolic equation (\ref{OPT_state}) is to find $y\in L^2(H^1(\Omega))\cap H^1(H^{-1}(\Omega))$ with $y|_{\Sigma}=u$ and $y(\cdot,0)|_\Omega=y_0(\cdot)$ such that
\begin{eqnarray}
\big(\frac{\partial y}{\partial t},v\big)+a(y, v)=(f,v)\ \ \mbox{a.a.}\ t\in (0,T],
\ \forall\ v\in H_0^1(\Omega).\label{stand_weak}
\end{eqnarray}
This setting requires $u\in L^2(H^{1\over 2}(\Gamma))$. Motivated by practical considerations (see e.g. the discussion in \cite{Berggren.M2004}) we are interested in controls $u\in U_{ad}$ defined in (\ref{control_set}). For a proper treatment of the state equation in this case we in the present paper use the transposition technique
introduced by Lions and Magenes (see \cite{Lions.JL;Magenes.E1972}) to argue the existence of a unique solution to the state equation (\ref{OPT_state}). The very weak form of (\ref{OPT_state}) that we shall utilize reads: Find $y\in L^2(L^2(\Omega))$ such that
\begin{eqnarray}
\int_{\Omega_T}y(-z_t-\Delta z)dxdt=-\int_{\Sigma}u\partial_{{n}}zdsdt+\int_{\Omega_T}fzdxdt+\int_\Omega y_0z(\cdot,0)dx\nonumber\\
\ \ \forall\ z\in L^2(H^2(\Omega)\cap H_0^1(\Omega))\cap H^1(L^2(\Omega))\label{weak}
\end{eqnarray}
with $z(\cdot,T)=0$ holds, where $\partial_{{n}}v:=\nabla v\cdot n$ with ${n}$
denoting the unit outward normal to $\Gamma$. Then the existence and uniqueness of a very weak solution of (\ref{weak}) is shown in the following lemma (see, e.g, \cite{Lions.JL;Magenes.E1972})
\begin{Lemma}\label{Lemma:state}
For $y_0\in L^2(\Omega)$, $f\in L^2(L^2(\Omega))$ and $u\in L^2(L^2(\Gamma))$, there exists a unique very weak solution $y\in L^2(L^2(\Omega))$ of (\ref{weak}) satisfying
\begin{eqnarray}
\|y\|_{L^2(L^2(\Omega))}\leq C\big(\|y_0\|_{L^2(\Omega)}+\|f\|_{L^2(L^2(\Omega))}+\|u\|_{L^2(L^2(\Gamma))}\big).\label{existence}
\end{eqnarray}
\end{Lemma}
\begin{proof}
For $y_0\in L^2(\Omega)$, $f\in L^2(L^2(\Omega))$ and $u\equiv 0$ it is straightforward to show that (\ref{OPT_state}) admits a unique solution $y\in L^2(H^1_0(\Omega))\cap H^1(H^{-1}(\Omega))$ in the sense of (\ref{stand_weak}), which also satisfies (\ref{existence}). To prove the lemma in the case $u\neq 0$ it is sufficient to consider the case $f\equiv 0$, $y_0\equiv 0$, where we follow the constructive approach of \cite{Deckelnick.K;Gunther.A;Hinze.M2009SICON}. For each $g\in L^2(L^2(\Omega))$ we denote by $z\in L^2(H^2(\Omega)\cap H_0^1(\Omega))\cap H^1(L^2(\Omega))$ the solution of \begin{equation}\label{parabolic_backward}
\left\{ \begin{aligned}-\frac{\partial z}{\partial t} -\Delta z=g \ \ &\mbox{in}\
\Omega_T, \\
 \ z=0  \ \ \ &\mbox{on}\ \Sigma,\\
 z(T)=0\ \ \ &\mbox{in}\ \Omega.
\end{aligned} \right.
\end{equation}
 Then we have $\partial_{{n}}z\in H^{\frac{1}{4}}(L^2(\Gamma))$ according to \cite{Lions.JL;Magenes.E1972}. Moreover, from the fact that $z\in L^2(H^2(\Omega))$ and $z=0$ on $\Sigma$ we obtain that $\partial_{{n}}z\in L^2(H^{1\over 2}(\Gamma))$ according to Lemma A.2 in \cite{Casas.E;Mateos.M;Raymond.JP2009}. We denote by $T:L^2(L^2(\Omega))\rightarrow L^2(L^2(\Gamma))$ the continuous linear operator which is defined by $Tg:=-\partial_{{n}}z|_{\Sigma}$ and denote its adjoint by $T^*$. Then with $y=T^*u$ we have
\begin{eqnarray}\nonumber
\int_{\Omega_T}ygdxdt=\int_{\Omega_T}y(-z_t-\Delta z)dxdt=\int_{\Omega_T}T^*ugdxdt=-\int_{\Sigma}u\partial_{{n}}zdsdt,
\end{eqnarray}
which verifies that $y$ satisfies (\ref{weak}). The estimate (\ref{existence}) follows
by observing that
$$|\int_{\Omega_T}ygdxdt|\leq C\|u\|_{L^2(L^2(\Gamma))}\|\partial_{{n}}z\|_{L^2(L^2(\Gamma))}\leq C\|u\|_{L^2(L^2(\Gamma))}\|g\|_{L^2(L^2(\Omega))}.$$
\end{proof}

Now we are ready to formulate the optimal control problem considered in the present paper. It reads
\begin{equation}\label{OPT_weak}
\left\{ \begin{aligned}
\mbox{min}\ \ J(y,u)={1\over 2}\|y-y_d\|^2_{L^2(L^2(\Omega))}+{\alpha\over 2}\|u\|^2_{L^2(L^2(\Gamma))}\\
\mbox{over}\ \ (y,u)\in {L^2(L^2(\Omega))}\times {L^2(L^2(\Gamma))}\\
\mbox{subject to}\ \ (\ref{weak})\ \mbox{and}\ u\in U_{ad}.
\end{aligned} \right.
\end{equation}
By standard arguments (see, e.g, \cite{Lions.JL1971}), there exists a unique solution $(y,u)$ for problem (\ref{OPT_weak}). Let $J(u):=J(y(u),u)$ denote the reduced cost functional, where for each $u\in L^2(L^2(\Gamma))$ the state $y(u)$ is the unique very weak solution of (\ref{weak}). Then $J$ is infinitely often Fr\'{e}chet differentiable. Moreover, the first order sufficient and necessary optimality conditions for problem (\ref{OPT_weak}) are given by
\begin{Theorem}
Assume that $u\in L^2(L^2(\Gamma))$ is the unique solution of problem (\ref{OPT_weak}) and let $y$ be the associated state. Then there exists a unique adjoint state $z\in L^2(H^1_0(\Omega))\cap H^1(H^{-1}(\Omega))$ such that
\begin{equation}\label{OPT_adjoint}
\left\{ \begin{aligned}-\frac{\partial z}{\partial t} -\Delta z=y-y_d \ \ &\mbox{in}\
\Omega_T, \\
 \ z=0  \ \ \ &\mbox{on}\ \Sigma,\\
 z(T)=0\ \ \ &\mbox{in}\ \Omega,
\end{aligned} \right.
\end{equation}
and
\begin{eqnarray}
J'(u)(v-u)=\int_0^T\int_{\Gamma}(\alpha u-\partial_{{ n}}z)(v-u)dsdt\geq 0,\ \ \forall\ v\in U_{ad}.\label{optimality}
\end{eqnarray}
\end{Theorem}
We note that (\ref{optimality}) is equivalent to
\begin{eqnarray}
J'(u)(v-u)=\int_{\Sigma}\alpha u(v-u)dsdt+\int_{\Omega_T}(y-y_d)(y(v)-y)dxdt\geq 0,\ \ \forall\ v\in U_{ad}\label{optimality1}
\end{eqnarray}
or
\begin{eqnarray}
u=P_{U_{ad}}\big(\frac{1}{\alpha}\partial_{{ n}}z\big),\label{projection}
\end{eqnarray}
where for each $v\in L^2(L^2(\Gamma))$, $y(v)$ is the solution of problem (\ref{weak}) with $u$ replaced by $v$, and $P_{U_{ad}}:L^2(L^2(\Gamma))\rightarrow U_{ad}$ denotes the orthogonal projection.

We now turn to the regularity properties of optimal controls $u$ on $\Sigma$. The proof of the following theorem can be found in, e.g, \cite{Kunisch.K;Vexler.B2007}.
\begin{Theorem}\label{reg_polygonal}
Let $(y,u,z)\in L^2(L^2(\Omega))\times L^2(L^2(\Gamma))\times L^2(H^1_0(\Omega))\cap H^1(H^{-1}(\Omega))$ be the solution of optimal control problem (\ref{OPT_weak})-(\ref{optimality1}), assume that $\Omega$ is a bounded, convex polygonal or polyhedral domain in $\mathbb{R}^n$, $f\in L^2(L^2(\Omega))$, $y_d\in L^2(L^2(\Omega))$ and $y_0\in L^2(\Omega)$. Then we have
\begin{eqnarray}
u\in L^2(H^{1\over 2}(\Gamma))\cap H^{\frac{1}{4}}(L^2(\Gamma)),\ \ \ y\in L^2(H^1(\Omega))\cap H^{\frac{1}{2}}(L^2(\Omega)),\label{polygoanl_regularity_yu}
\end{eqnarray}
and
\begin{eqnarray}
 z\in L^2(H^2(\Omega)\cap H_0^1(\Omega))\cap H^1(L^2(\Omega)).\label{polygoanl_regularity_p}
\end{eqnarray}
\end{Theorem}
\begin{proof}
From $f\in L^2(L^2(\Omega))$, $y_0\in L^2(\Omega)$ and $u\in L^2(L^2(\Gamma))$ we conclude that $y\in L^2(L^2(\Omega))$ according to Lemma \ref{Lemma:state}. Thus, $y_d\in L^2(L^2(\Omega))$ implies $z\in L^2(H^2(\Omega)\cap H_0^1(\Omega))\cap H^1(L^2(\Omega))$, which in turn implies $\partial_n z\in L^2(H^{1\over 2}(\Gamma))\cap H^{\frac{1}{4}}(L^2(\Gamma))$ (see \cite{Casas.E;Mateos.M;Raymond.JP2009}). From (\ref{projection})  we obtain that  $u\in L^2(H^{1\over 2}(\Gamma))\cap H^{\frac{1}{4}}(L^2(\Gamma))$ and thus $y\in L^2(H^1(\Omega))\cap H^{\frac{1}{2}}(L^2(\Omega))$ (see \cite{Lions.JL;Magenes.E1972}). This completes the proof.
\end{proof}

Although we only consider the finite element approximation to optimal control problems posed on polygonal or polyhedral domains, we note that for optimal control problems posed on curved domains with smooth boundary we can get higher regularities (see e.g. \cite{Kunisch.K;Vexler.B2007}).
\begin{Theorem}\label{reg_curve}
Let $(y,u,z)\in L^2(L^2(\Omega))\times L^2(L^2(\Gamma))\times L^2(H^1(\Omega))\cap H^1(H^{-1}(\Omega))$ be the solution of optimal control problem (\ref{OPT_weak})-(\ref{optimality1}), assume that $\Omega$ is a bounded domain in $\mathbb{R}^n$ with sufficiently smooth boundary $\Gamma$, $f\in L^2(L^2(\Omega))$, $y_d\in L^2(H^1(\Omega))\cap H^{\frac{1}{2}}(L^2(\Omega))$ and $y_0\in H^{\frac{1}{2}-\epsilon}(\Omega)$. Then we have
\begin{eqnarray}
u\in L^2(H^1(\Gamma))\cap H^{\frac{1}{2}}(L^2(\Gamma)),\ \ \ y\in L^2(H^{\frac{3}{2}-\epsilon}(\Omega))\cap H^{\frac{3-2\epsilon}{4}}(L^2(\Omega))\label{regularity_yu}
\end{eqnarray}
and
\begin{eqnarray}
 z\in L^2(H^3(\Omega)\cap H_0^1(\Omega))\cap H^{{3\over 2}}(L^2(\Omega))\label{regularity_p}
\end{eqnarray}
for every $\epsilon \in (0,\frac{1}{2}]$.
\end{Theorem}
\begin{proof}
From Theorem \ref{reg_polygonal} we already have $z\in L^2(H^2(\Omega)\cap H_0^1(\Omega))\cap H^1(L^2(\Omega))$, $u\in L^2(H^{1\over 2}(\Gamma))\cap H^{\frac{1}{4}}(L^2(\Gamma))$ and $y\in L^2(H^1(\Omega))\cap H^{\frac{1}{2}}(L^2(\Omega))$. If $y_d\in L^2(H^1(\Omega))\cap H^{\frac{1}{2}}(L^2(\Omega))$, since the boundary $\Gamma$ of domain $\Omega$ is sufficiently smooth, we from \cite{Lions.JL;Magenes.E1972} have the improved regularity $z\in L^2(H^{3}(\Omega))\cap H^{3\over 2}(L^2(\Omega))$. The trace theorem for parabolic equations implies $\partial_n z\in L^2(H^{3\over 2}(\Gamma))\cap H^{\frac{3}{4}}(L^2(\Gamma))$. From (\ref{projection}) we then conclude $u\in L^2(H^1(\Gamma))\cap H^{\frac{1}{2}}(L^2(\Gamma))$. This implies that $y\in L^2(H^{\frac{3}{2}-\epsilon}(\Omega))\cap H^{\frac{3-2\epsilon}{4}}(L^2(\Omega))$ for each $\epsilon >0$, see, e.g., \cite{Lions.JL;Magenes.E1972}, which completes the proof.
\end{proof}

In our analysis we frequently use results of the following backward in time parabolic problem:
\begin{eqnarray}\label{auxiliary}
\left\{\begin{aligned}
-w_t-\Delta w=g\ \ \ \mbox{in}\ \Omega_T,\\
w=0\ \ \mbox{on}\ \Sigma,\\
w(T)=0\ \ \mbox{in}\ \Omega.
\end{aligned}\right.
\end{eqnarray}
If $g\in L^2(L^2(\Omega))$, then (\ref{auxiliary}) has a unique solution $w\in L^2(H^2(\Omega)\cap H_0^1(\Omega))\cap H^1(L^2(\Omega))$ satisfying
\begin{eqnarray}
\|w\|_{L^2(H^2(\Omega))}+\|w_t\|_{L^2(L^2(\Omega))}\leq C\|g\|_{L^2(L^2(\Omega))},\label{adjoint_stability}\\
\|w(0)\|_{1,\Omega}\leq C\|g\|_{L^2(L^2(\Omega))}.\label{initial_stability}
\end{eqnarray}


\section{Finite element discretization of the state equation and optimal control problems} \setcounter{equation}{0}
At first let us consider the finite element approximation of the state equation (\ref{OPT_state}). For the spatial discretization we consider conforming Lagrange triangular
elements.

We assume that $\Omega$ is a polygonal domain. Let $\mathscr{T}^h$ be a quasi-uniform partitioning of
$\Omega$ into disjoint regular triangles $\tau$, so that
$\bar{\Omega}=\bigcup_{\tau\in \mathscr{T}^h}\bar{\tau}$.
Associated with $\mathscr{T}^h$ is a finite dimensional
subspace $V^h$ of $C(\bar\Omega)$, such that for $\chi\in V^h$ and $\tau\in \mathscr{T}^h$, $\chi|_{\tau}$ are
piecewise linear polynomials. We set $V^h_0=V^h\cap H_0^1(\Omega)$.

Let $\mathscr{T}_U^h$ be a partitioning of $\Gamma$ into disjoint
regular segments $s$, so that
$\Gamma=\bigcup_{s\in \mathscr{T}_U^h}\bar s$. Associated with
$\mathscr{T}_U^h$ is another finite dimensional subspace $U^h$ of
$L^2(\Gamma)$, such that for $\chi\in U^h$ and $s\in \mathscr{T}_U^h$, $\chi|_s$ are piecewise linear polynomials. Here
 we suppose that $\mathscr{T}_U^h$ is the restriction of $\mathscr{T}^h$ on the boundary $\Gamma$ and $U^h=V^h(\Gamma)$, where $V^h(\Gamma)$ is the restriction of $V^h$ on the boundary $\Gamma$. 

For the standard Lagrange interpolation operator $I_h:C(\bar\Omega)\rightarrow V^h$, we have the following error estimate (see, e.g., \cite{Ciarlet.PG1978})
\begin{eqnarray}
\|w-I_hw\|_{l,\Omega}\leq Ch^{m-l}\|v\|_{m,\Omega},\ \ \ 0\leq l\leq 1\leq m\leq 2.\label{interpolation_error}
\end{eqnarray}
To define our discrete scheme, we need to introduce some projection operators. Here $Q_h:L^2(\Gamma)\rightarrow V^h(\Gamma)$ and $\tilde Q_h:L^2(\Omega)\rightarrow V^h_0$ denote the orthogonal projection operators. Furthermore, $R_h: H^1(\Omega)\rightarrow V^h_0$ denotes the Ritz projection operator defined as
\begin{eqnarray}
a(R_hw,v_h)=a(w,v_h),\ \ \forall\ v_h\in V_0^h.\label{elliptic_proj}
\end{eqnarray}
It is well known that the Ritz projection satisfies
\begin{eqnarray}
\|w-R_hw\|_{s,\Omega}\leq Ch^{l-s}\|w\|_{l,\Omega},\ \ w\in H_0^1(\Omega)\cap H^l(\Omega),\ \forall\ 0\leq s\leq 1\leq l\leq 2.\label{Ritz_error}
\end{eqnarray}
For the $L^2(\Gamma)$ projection operator $Q_h$ we also have
\begin{eqnarray}
\|w-Q_hw\|_{0,\Gamma}\leq Ch^{s-{1\over 2}}\|w\|_{s,\Omega}\ \ \mbox{for}\ w\in H^s(\Omega), \ \frac{1}{2}\leq s\leq 2,\label{boundary_trace}
\end{eqnarray}
and
\begin{eqnarray}
\|(I-Q_h)\partial_{{n}}w\|_{0,\Gamma}\leq Ch^{{1\over 2}}\|w\|_{2,\Omega}\ \ \mbox{for}\ w\in H^2(\Omega).\label{boundary_normal}
\end{eqnarray}

In our following analysis we need estimates for discrete harmonic functions.
\begin{Lemma}
Let $v_h\in V^h(\Gamma)$, and suppose that $w\in H^1(\Omega)$ is the solution of
\begin{eqnarray}
a(w,\phi)=0,\ \ \ \forall\ \phi\in H^1_0(\Omega),\ \ \ w=v_h \ \mbox{on}\ \Gamma\label{cont_harmonic}
\end{eqnarray}
and $w_h\in V^h$ is the solution of
\begin{eqnarray}
a(w_h,\phi_h)=0,\ \ \ \forall\ \phi_h\in V^h_0,\ \ \ w_h=v_h \ \mbox{on}\ \Gamma.\label{disc_harmonic}
\end{eqnarray}
Then
\begin{eqnarray}
\|w-w_h\|_{1,\Omega}\leq C\|v_h\|_{{1\over 2},\Gamma}\leq Ch^{-{1\over 2}}\|v_h\|_{0,\Gamma},\label{harmonic_error}\\
\|w_h\|_{0,\Omega}+h^{-{1\over 2}}\|w_h\|_{1,\Omega}\leq C\| v_h\|_{0,\Gamma},\label{discrete_harmonic}\\
\|w-w_h\|_{{1\over 2},\Omega}\leq C\|v_h\|_{0,\Gamma}.\label{Lipschitz_h}
\end{eqnarray}
\end{Lemma}
\begin{proof}
The proof of (\ref{harmonic_error}) and (\ref{discrete_harmonic}) can be found in \cite{Bramble.JH;Pasciak.JE;Schatz.AH1986}, \cite{Casas.E;Raymond.JP2006} and \cite{French.DA;King.JT1991}. Here we provide a proof of (\ref{Lipschitz_h}). For each $g\in H^{-{1\over 2}}(\Omega)$ let $\psi_g\in H^{3\over 2}(\Omega)\cap H^1_0(\Omega)$ be the solution of
\begin{equation}\label{auxi}
a(\phi,\psi_g)=\langle g,\phi\rangle_{H^{-{1\over 2}}, H^{1\over 2}},\ \ \ \forall\ \phi\in H^1_0(\Omega).
\end{equation}
Then we have $\|\psi_g\|_{{3\over 2},\Omega}\leq C\|g\|_{-{1\over 2},\Omega}$. Note that from (\ref{cont_harmonic}) and (\ref{disc_harmonic}) we have
\begin{eqnarray}
\nonumber
\langle g,w-w_h\rangle_{H^{-{1\over 2}}, H^{1\over 2}}= a(w-w_h,\psi_g)= a(w-w_h,\psi_g-I_h\psi_g),
\end{eqnarray}
where $I_h\psi_g$ is the linear Lagrange interpolation of $\psi_g$ (\cite{Ciarlet.PG1978}). Then standard error estimates lead to
\begin{eqnarray}
\nonumber
\langle g,w-w_h\rangle_{H^{-{1\over 2}}, H^{1\over 2}}
&=& a(w-w_h,\psi_g-I_h\psi_g)\nonumber\\
&\leq&\|w-w_h\|_{1,\Omega}\|\psi_g-I_h\psi_g\|_{1,\Omega}\nonumber\\
&\leq& Ch^\frac{1}{2}\|v_h\|_{{1\over 2},\Gamma}\|\psi_g\|_{\frac{3}{2},\Omega} \\
\nonumber
&\leq& C\|v_h\|_{0,\Gamma}\|g\|_{-\frac{1}{2},\Omega},
\end{eqnarray}
where we have used the estimate (\ref{harmonic_error}). This implies
\begin{eqnarray}
\|w_h-w\|_{\frac{1}{2},\Omega}\leq C\|v_h\|_{0,\Gamma},\nonumber
\end{eqnarray}
which proves (\ref{Lipschitz_h}).
\end{proof}

The semi-discrete finite element approximation of
(\ref{OPT_state}) reads: Find $y_h\in L^2(V^h)$ such that
\begin{equation}\label{semidiscrete}
\left\{\begin{aligned}
&-(y_h,\partial_tv_h)_{\Omega_T} +a(y_h,v_h)_{\Omega_T}=(f,v_h)_{\Omega_T} + (y_0^h, v_h(\cdot,0))\ \ \forall\ v_h\in H^1(V^h_0),\\
&\ y_h= Q_h(u)\ \ \mbox{on}\ \Sigma
\end{aligned}\right.
\end{equation}
with $v_h(\cdot, T)=0$, $y_0^h\in V^h$ an approximation of $y_0$ using the $L^2$-projection, and $Q_h$ the projection operator from $L^2(\Gamma)$ to $V^h(\Gamma)$. Note that the above semi-discrete scheme is well-defined and admits a unique solution since $Q_h(u)\in L^2(H^{1\over 2}(\Gamma))$.

The semi-discrete finite element approximation of
(\ref{OPT})-(\ref{OPT_state}) reads as follows:
\begin{eqnarray}
\min\limits_{u_h\in U_{ad}^h,y_h\in L^2(V^h)} J_h(y_h,u_h)&={1\over 2}\|y_h-y_d\|^2_{L^2(L^2(\Omega))}+{\alpha\over 2}\|u_h\|^2_{L^2(L^2(\Gamma))}\label{OPT_semi}
\end{eqnarray}
subject to
\begin{equation}\label{OPT_state_semi}
\left\{ \begin{aligned}
&-(y_h,\partial_tv_h)_{\Omega_T} +a(y_h,v_h)_{\Omega_T}=(f,v_h)_{\Omega_T} + (y_0^h, v_h(\cdot,0))\ \ \forall\ v_h\in H^1(V^h_0),\\
&\ y_h= Q_h(u_h)\ \ \mbox{on}\ \Sigma,
\end{aligned}\right.
\end{equation}
where $y_h\in L^2(V^h)$, $y_0^h\in V^h$ is an approximation
of $y_0$, and $U_{ad}^{h}$ is an appropriate approximation to $U_{ad}$ depending on the discretization scheme for the control.

It follows that the control problem (\ref{OPT_semi})-(\ref{OPT_state_semi}) has a
unique solution $(y_h,u_h)$ and that a pair $(y_h,u_h)$ is the
solution of the problem (\ref{OPT_h})-(\ref{OPT_state_h}) if and only if
there is a co-state $z_h\in L^2(V^h_0)$ such that the triplet $(y_h,z_h,u_h)$
satisfies the following optimality conditions:
\begin{equation}\label{state_semi}
\left\{\begin{aligned}
&-(y_h,\partial_tv_h)_{\Omega_T} +a(y_h,v_h)_{\Omega_T}=(f,v_h)_{\Omega_T} + (y_0^h, v_h(\cdot,0))\ \ \forall\ v_h\in H^1(V^h_0),\\
&\ y_h= Q_h(u_h)\ \ \mbox{on}\ \Sigma,
\end{aligned}\right.
\end{equation}

\begin{equation}\label{adjoint_semi}
\left\{\begin{aligned}
-\big(\frac{\partial z_h}{\partial t},q_h\big)+a(q_h,z_h)=(y_h-y_{d},q_h),\ \
\forall\ q_h\in V^h_0,\\
z_h=0\ \ x\in \Sigma;\ z_h(x,T)=0\ \ \ x\in\Omega,
\end{aligned}\right.
\end{equation}

\begin{equation}
\int_0^T\int_{\Omega}(y_h-y_{d})(y_h(v_h)-y_h)dxdt+\alpha\int_0^T\int_{\Gamma}u_h(v_h-u_h)dsdt\geq 0, \ \ \forall\ v_h\in U_{ad}^h,\label{adjoint_control_semi}
\end{equation}
where $y_h(v_h)\in L^2(V^h)$ is the solution of state equation (\ref{state_semi}) with Dirichlet boundary condition $Q_h(v_h)$.

We next consider the fully discrete approximation for above
semi-discrete problem by using the dG(0) scheme in time. For simplicity we consider an equi-distant partition of the time interval. Let $0=t_0<t_1<\cdots <t_{N-1}<t_N=T$ with $k={T\over N}$ and $t_i=ik$,
 $i=1,2,\cdots,N$. We also set $I_i:=(t_{i-1},t_i]$. For
$i=1,2,\cdots,N$, we construct the finite element spaces $V^h\in
H^1(\Omega)$ with the mesh $\mathscr{T}^h$. Similarly,
we construct the finite element spaces $U^h\in L^2(\Gamma)$ with the mesh $\mathscr{T}^h_U$. In our case we have $U^h=V^h(\Gamma)$. Then we denote by $V^h$ and $U^h$ the finite element spaces defined on $\mathscr{T}^h$ and $\mathscr{T}^h_U$ on each time step.

Let $V_k$ denote the space of piecewise constant functions on the time partition. We define the $L^2$ projection operator $P_k:L^2(0,T)\rightarrow V_k$ on $I_i$ through
\begin{eqnarray}
(P_kw)(t)=\frac{1}{k}\int_{I_i}w(s)ds\ \ \mbox{for}\ t\in I_i\ (i=1,\cdots,N).\nonumber
\end{eqnarray}
Then we have the following estimate
\begin{eqnarray}
\|(I-P_k)w\|_{L^2(0,T;H)}\leq Ck\|w_t\|_{L^2(0,T;H)},\ \ \forall\ w\in H^1(0,T;H),\label{Pk_error}
\end{eqnarray}
where $H$ denotes some separable Hilbert space .

We consider a dG(0) scheme for the time discretization and set
\begin{eqnarray}
V_{hk}:=\Big\{\phi:\overline\Omega\times [0,T]\rightarrow \mathbb{R},\ \phi(\cdot,t)|_{\overline\Omega}\in V^h,\ \phi(x,\cdot)|_{I_n}\in \mathbb{P}_0\ \mbox{for }\ n=1,\cdots,N\Big\},\nonumber
\end{eqnarray}
i.e. $\phi\in V_{hk}$ is a piecewise constant polynomial w.r.t. time. We also set $V_{hk}(\Gamma)$ as the restriction of $V_{hk}$ on $L^2(L^2(\Gamma))$. We set $Q=Q_hP_k=P_kQ_h$. Thus, we have $Q:L^2(L^2(\Gamma))\rightarrow V_{hk}(\Gamma)$. For $Y,\Phi\in V_{hk}$ we set
\begin{eqnarray}
A(Y,\Phi):=\sum\limits_{n=1}^N ka(Y^n,\Phi^n)+\sum\limits_{n=2}^N(Y^n-Y^{n-1},\Phi^n)+(Y_+^0,\Phi_+^0),\nonumber
\end{eqnarray}
where $\Phi^n:=\Phi^n_-$, $\Phi^n_{\pm}=\lim_{s\rightarrow 0^\pm}\Phi(t_n+s)$.

For each $u\in L^2(L^2(\Gamma))$ the fully discrete dG(0)-cG(1) finite element approximation of (\ref{semidiscrete}) now reads: Find $Y_{hk}\in V_{hk}$ such that
\begin{equation}\label{full_discrete}
\left\{\begin{aligned}
A(Y_{hk},\Phi)=( f,\Phi)_{\Omega_T}+(y_0,\Phi_+^0),\ \ &\forall\ \Phi\in V_{hk}^0,\\
 Y_{hk}=Q(u)\ \ &\mbox{on}\ \Gamma,
 \end{aligned}\right.
 \end{equation}
 where $V_{hk}^0$ denotes the subspace of $V_{hk}$ with functions vanishing on the boundary $\Gamma$.

It is easy to see that on each time interval $I_i$, $Y_{hk}^i\in V^h$ solves the following problem:
\begin{equation}\label{full_discrete1}
\left\{\begin{aligned}
&\big(\frac{Y_{hk}^i-Y_{hk}^{i-1}}{k},w_h\big)+a(Y_{hk}^i,w_h)=(P_k^if,w_h),\
\forall\ w_h\in V^h_0,\;\; i=1,\cdots,N,\\
& Y_{hk}^0(x)=y_0^h(x),\ x\in\Omega; \ \ Y_{hk}^i= Q_h(P_k^iu),\ i=1,\cdots,N\ \ \mbox{on}\ \Gamma.
\end{aligned}\right.
\end{equation}
Here we use $L^2$-projection to approximate the non-smooth Dirichlet boundary condition in (\ref{full_discrete}).

In the following we need to investigate the stability behavior of the fully discrete scheme (\ref{full_discrete}) with respect to  the initial value $y_0$, the righthand side $f$ and the Dirichlet boundary conditions $u$.
\begin{Lemma}\label{stab}
There exists a constant $C$ independent of $h,k$ and the data $(f,y_0)$ such that for  $k=O(h^2)$, there holds
\begin{eqnarray}
&&\sum\limits_{i=1}^{N}\Big(\| Y_{hk}^i- Y_{hk}^{i-1}\|_{0,\Omega}^2+k\|Y_{hk}^i\|_{1,\Omega}^2\Big)+\| Y_{hk}^N\|_{0,\Omega}^2\nonumber\\
&\leq& C\big(\|y_0\|^2_{0,\Omega}+h^{-1}\|u\|_{L^2(L^2(\Gamma))}^2+\|f\|_{L^2(L^2(\Omega))}^2\big).\label{stability}
\end{eqnarray}
\end{Lemma}

\begin{proof}
Since the problem under consideration is linear it is sufficient to consider the problems with either $f\equiv 0$, $y_0\equiv 0$ or $u\equiv 0$.

Let us first assume that $f\equiv 0$, $y_0\equiv 0$. The proof follows the idea of \cite{French.DA;King.JT1993}. To begin with we introduce the following problem: Find $y_u\in V_{hk}$ with
\begin{equation}\label{y_u}
\left\{\begin{aligned}
&\big(\frac{y_u^i-y_u^{i-1}}{k},w_h\big)+a(y_u^i,w_h)=0,\
\forall\ w_h\in V^h_0,\;\; i=1,\cdots,N,\\
&y_u^0(x)=0\ x\in\Omega; \ \ y_u^i= Q_h(P_k^iu),\ i=1,\cdots,N\ \ \mbox{on}\ \Gamma.
\end{aligned}\right.
\end{equation}
For arbitrary $y_h\in V^h$ we have the splitting
\begin{eqnarray}
y_h=y_1+R_hy_h\ \ \ \ \mbox{and}\ \ y_h=y_2+\tilde Q_hy_h,\nonumber
\end{eqnarray}
where $\tilde Q_hy_h\in V_0^h$ and $R_hy_h\in V_0^h$ are the $L^2$-projection and Ritz-projection of $y_h$, respectively. Then we have $y_2|_{\Gamma}=y_h$, $y_1|_{\Gamma}=y_h$ and
\begin{eqnarray}
(y_2,v_h)=0\ \ \ \ \mbox{and}\ \ a(y_1,v_h)=0,\ \ \forall\ v_h\in V_0^h.\nonumber
\end{eqnarray}
Let $y^i_u=y_2^i+\tilde Q_hy_u^i$. Then (\ref{y_u}) delivers
\begin{eqnarray}
(\tilde Q_hy^i_u-\tilde Q_hy^{i-1}_u,w_h)+ka(\tilde Q_hy^i_u,w_h)=-ka(y_2^i,w_h),\ \ \forall\ w_h\in V^h_0.
\end{eqnarray}
Let $w_h=\tilde Q_hy^i_u$. Summation  from 1 to $N$ then gives
\begin{eqnarray}
\sum\limits_{i=1}^{N}\Big(\frac{1}{2}\|\tilde Q_hy^i_u-\tilde Q_hy^{i-1}_u\|_{0,\Omega}^2+ka(\tilde Q_hy^i_u,\tilde Q_hy^i_u)\Big)+\frac{1}{2}\|\tilde Q_hy^N_u\|_{0,\Omega}^2=-\sum\limits_{i=1}^{N}ka(y_2^i, \tilde Q_hy^i_u),\nonumber
\end{eqnarray}
which implies
\begin{eqnarray}
\frac{1}{2}\sum\limits_{i=1}^{N}\Big(\|\tilde Q_hy^i_u-\tilde Q_hy^{i-1}_u\|_{0,\Omega}^2+ka(\tilde Q_hy^i_u,\tilde Q_hy^i_u)\Big)+\frac{1}{2}\|\tilde Q_hy^N_u\|_{0,\Omega}^2\leq \frac{1}{2}\sum\limits_{i=1}^{N}ka(y_2^i, y_2^i).\label{Ry}
\end{eqnarray}
On the other hand, we have
\begin{eqnarray}
\sum\limits_{i=1}^{N}\big(\|y^i_2-y^{i-1}_2\|_{0,\Omega}^2+ka(y^i_2,y^i_2)\big)+\|y^N_2\|_{0,\Omega}^2\leq \sum\limits_{i=1}^{N}\big(ka(y_2^i, y_2^i)+2\|y^i_2\|_{0,\Omega}^2\big).\label{y_0}
\end{eqnarray}
Combining (\ref{Ry}) and (\ref{y_0}) yields
\begin{eqnarray}
\sum\limits_{i=1}^{N}\big(\|y^i_u-y_u^{i-1}\|_{0,\Omega}^2+ka( y_u^i,y_u^i)\big)+\| y_u^N\|_{0,\Omega}^2\leq C\sum\limits_{i=1}^{N}\big(ka(y_2^i, y_2^i)+\|y^i_2\|_{0,\Omega}^2\big).
\end{eqnarray}
For $y^i_u\in V^h$ we also have the splitting $y^i_u=y_1^i+R_hy_u^i$. Since $C_0^\infty(\Omega)$ is dense in $H^s(\Omega)$ for $s\leq {1\over 2}$ (see \cite{French.DA;King.JT1993}), there exists $\{v^n\}\subset C_0^\infty(\Omega)$ such that $\lim\limits_{n\rightarrow \infty}\|v^n-y_1^i\|_{{1\over 2},\Omega}=0$. From the definition of the $L^2$-projection we conclude that
 \begin{eqnarray}
 y_2^i&=&(I-\tilde Q_h)y_u^i=(I-\tilde Q_h)(y_1^i+R_hy_u^i)\nonumber\\
 &=&(I-\tilde Q_h)y_1^i,\nonumber
 \end{eqnarray}
 thus, we deduce that
\begin{eqnarray}
\|y_2^i\|_{0,\Omega}=\|(I-\tilde Q_h)y_1^i\|_{0,\Omega}=\lim_{n\rightarrow \infty}\|(I-\tilde Q_h)v^n\|_{0,\Omega}\leq Ch^{1\over 2}\lim_{n\rightarrow \infty}\|v^n\|_{{1\over 2},\Omega}= Ch^{1\over 2}\|y_1^i\|_{{1\over 2},\Omega}.\nonumber
\end{eqnarray}
We note that $y_1^i|_{\Gamma}=y^i_u=Q_h(P_k^iu)$ and
\begin{eqnarray}
 a(y_1^i,\phi_h)=0,\ \ \forall\ \phi_h\in V_0^h.\nonumber
\end{eqnarray}
Let $w^i\in H^1(\Omega)$ be the solution of (\ref{cont_harmonic}) with $v_h$ substituted by $Q_h(P_k^iu)$. Then $\|w^i\|_{{1\over 2},\Omega}\leq C\|Q_h(P_k^iu)\|_{0,\Gamma}$ and $y_1^i$ is the finite element approximation to $w^i$. So we deduce from Lemma 3.1 that
\begin{eqnarray}
\|y_1^i\|_{{1\over 2},\Omega}&\leq& \|y_1^i-w^i\|_{{1\over 2},\Omega}+\|w^i\|_{{1\over 2},\Omega}\nonumber\\
&\leq&C\|Q_h(P_k^iu)\|_{0,\Gamma},\nonumber
\end{eqnarray}
which in turn gives
\begin{eqnarray}
\|y_2^i\|_{0,\Omega}\leq Ch^{1\over 2}\|Q_h(P_k^iu)\|_{0,\Gamma}.\nonumber
\end{eqnarray}
Inverse estimates also yield
\begin{eqnarray}
\|y_2^i\|_{1,\Omega}\leq Ch^{-{1\over 2}}\|Q_h(P_k^iu)\|_{0,\Gamma}.\nonumber
\end{eqnarray}
With the help of above estimates and norm interpolation we are led to
\begin{eqnarray}
\|y_2^i\|_{s,\Omega}\leq Ch^{{1\over 2}-s}\|Q_h(P_k^iu)\|_{0,\Gamma},\ \ 0\leq s\leq 1,\ \ i=1,2,\cdots,N.
\end{eqnarray}
Thus, for $k=O(h^2)$ we have
\begin{eqnarray}
&&\sum\limits_{i=1}^{N}\big(\|y^i_u-y_u^{i-1}\|_{0,\Omega}^2+ka( y_u^i,y_u^i)\big)+\| y_u^N\|_{0,\Omega}^2\nonumber\\
&\leq& C\sum\limits_{i=1}^{N}\big (ka(y_2^i, y_2^i)+\|y^i_2\|_{0,\Omega}^2\big )\nonumber\\
&\leq&C\sum\limits_{i=1}^{N}\big(kh^{-1}\|Q_h(P_k^iu)\|_{0,\Gamma}^2+h\|Q_h(P_k^iu)\|_{0,\Gamma}^2\big)\nonumber\\
&\leq&Ch^{-1}\sum\limits_{i=1}^{N}\int_{I_i}\|Q_h(P_k^iu)\|_{0,\Gamma}^2\nonumber\\
&\leq&Ch^{-1}\| Q(u)\|_{L^2(L^2(\Gamma))}^2\leq Ch^{-1}\|u\|_{L^2(L^2(\Gamma))}^2.
\end{eqnarray}
This gives
\begin{eqnarray}
&&\sum\limits_{i=1}^{N}\big(\|y^i_u-y_u^{i-1}\|_{0,\Omega}^2+k\| y_u^i\|_{1,\Omega}^2\big)+\| y_u^N\|_{0,\Omega}^2\leq Ch^{-1}\|u\|_{L^2(L^2(\Gamma))}^2.\label{stability_u}
\end{eqnarray}

 For the case $u\equiv 0$, let $y_f\in L^2(H_0^1(\Omega))\cap H^{1}(H^{-1}(\Omega))$ be the solution of following problem
 \begin{eqnarray}\label{y^f_semidiscrete}
\left\{\begin{aligned}
 &\big(\frac{\partial y_f}{\partial
 t},w\big)+a(y_f,w)=(f,w),\ \ \forall\ w\in
 H_0^1(\Omega),\;\;t\in(0,T],\\
 &y_f(x,0)=y_0 \ \ x\in\Omega;\ \ y_f= 0\ \ \mbox{on}\ \Gamma.
 \end{aligned}\right.
\end{eqnarray}
 Then we have
 \begin{eqnarray}
 \|y_f\|_{L^2(H^1(\Omega))}+\|\frac{\partial y_f}{\partial t}\|_{L^2(H^{-1}(\Omega))}\leq C\big(\|f\|_{L^2(L^2(\Omega))}+\|y_0\|_{0,\Omega}\big).\label{y^f_regularity}
 \end{eqnarray}
Let $y_f^i\in V^h$, $i=1,2,\cdots, N$ be the solutions of following problems:
\begin{eqnarray}\label{y_f}
\left\{\begin{aligned}
&\big(\frac{y_f^i-y_f^{i-1}}{k},w_h\big)+a(y_f^i,w_h)=(P_k^if,w_h),\
\forall\ w_h\in V^h_0,\;\; i=1,\cdots,N,\\
&y_f^0(x)=y_0^h(x),\ x\in\Omega; \ \ y_f^i= 0,\ i=1,\cdots,N,\ \ \mbox{on}\ \Gamma.
\end{aligned}\right.
\end{eqnarray}
Then we have $Y_{hk}^i= y_f^i+y_u^i$, since $ y_f^i$ is the standard fully discrete approximation of $y_f$. It is straightforward to prove that
\begin{eqnarray}
\sum\limits_{i=1}^{N}\Big(\|y_f^i-y_f^{i-1}\|_{0,\Omega}^2+k\|y_f^i\|_{1,\Omega}^2\Big)+\|y_f^N\|_{0,\Omega}^2\leq C\big(\|y_0\|_{0,\Omega}^2 + \|f\|_{L^2(L^2(\Omega))}^2\big).\label{y_f_estimate}
\end{eqnarray}
Combining (\ref{stability_u}) and (\ref{y_f_estimate}) completes the proof.

\end{proof}

We next consider the fully discrete approximation for above
semi-discrete optimal control problems by using the dG(0) scheme in time. The fully discrete approximation scheme
of (\ref{OPT_semi})-(\ref{OPT_state_semi}) is to find $(Y_{hk},U_{hk})\in
V_{hk}\times U_{ad}^{hk}$, such that
\begin{equation}
\min\limits_{U_{hk}\in
U^{hk}_{ad},Y_{hk}\in V_{hk}}J_{hk}(Y_{hk},U_{hk})=\sum\limits^{N}_{i=1}k\bigg\{{ 1\over 2}\int_{\Omega}(Y_{hk}^i-P_k^i y_d)^2dx+{\alpha\over 2}\int_{\Gamma}(U_{hk}^i)^2ds\bigg\}\label{OPT_h}
\end{equation}
subject to
\begin{equation}\label{OPT_state_h}
\left\{\begin{aligned}
A(Y_{hk},\Phi)=(f,\Phi)_{\Omega_T}+(y_0,\Phi_+^0),\ \ &\forall\ \Phi\in V_{hk}^0,\\
Y_{hk}=Q(U_{hk})\ \ &\mbox{on}\ \Sigma.
\end{aligned}\right.
\end{equation}
Here $U_{ad}^{hk}$ is an appropriate approximation to $U_{ad}$. We set $U_{ad}^{hk}= V_{hk}(\Gamma)\cap U_{ad}$ for the full discretization of the control problem (\ref{OPT})-(\ref{OPT_state}) and $U_{ad}^{hk}\equiv U_{ad}$ for its variational discretization.

It follows from standard arguments (see \cite{Lions.JL1971}) that the above control problem has a unique solution
$(Y_{hk},U_{hk})$, and that a pair $(Y_{hk},U_{hk})\in
V_{hk}\times U_{ad}^{hk}$ is the solution of
(\ref{OPT_h})-(\ref{OPT_state_h}) if and only if there is a co-state
$Z_{hk}\in V_{hk}^0$, such that the triplet
$(Y_{hk},Z_{hk},U_{hk})\in V_{hk}\times V_{hk}^0\times U_{ad}^{hk}$ satisfies the following optimality conditions:
\begin{equation}\label{state_full}
\left\{\begin{aligned}
A(Y_{hk},\Phi)=(f,\Phi)_{\Omega_T}+(y_0,\Phi_+^0),\ \ &\forall\ \Phi\in V_{hk}^0,\\
Y_{hk}=Q(U_{hk})\ \ &\mbox{on}\ \Sigma,
\end{aligned}\right.
\end{equation}
\begin{equation}\label{adjoint_full}
\left\{\begin{aligned}
A(\Phi,Z_{hk})=\sum\limits^{N}_{i=1}\int_{I_i}(Y_{hk}^i-y_d,\Phi)dt,\ \ &\forall\ \Phi\in V_{hk}^0,\\
Z_{hk}=0\ \ &\mbox{on}\ \Sigma,
\end{aligned}\right.
\end{equation}
\begin{equation}
\int_{\Omega_T}(Y_{hk}-y_{d})(Y_{hk}(v_{hk})-Y_{hk})dxdt+\alpha\int_{\Sigma}U_{hk}(v_{hk}-U_{hk})dsdt\geq 0,\ \forall\ v_{hk}\in U_{ad}^{hk},\label{adjoint_control_full}
\end{equation}
where $Y_{hk}(v_{hk})$ is the solution of problem (\ref{state_full}) with Dirichlet boundary conditions $Q(v_{hk})$.

To derive an expression for the derivative of $J_{hk}:L^2(L^2(\Gamma))\rightarrow \mathbb{R}$ analogous to the one of $J$ given by formula (\ref{optimality}) we have to define a discrete normal derivative $\partial^{hk}_nZ_{hk}\in V_{hk}(\Gamma)$ satisfying \begin{eqnarray}
\int_{\Sigma}\partial^{hk}_nZ_{hk}\Phi dsdt&=&A(\Phi,Z_{hk})-\int_{\Omega_T}(Y_{hk}-y_{d})\Phi dxdt,\ \ \forall\ \Phi\in V_{hk}.\label{reconstruction}
\end{eqnarray}
It is easy to verify that the linear form
\begin{eqnarray}
L(\Phi):=A(\Phi,Z_{hk})-\int_{\Omega_T}(Y_{hk}-y_{d})\Phi dxdt\nonumber
\end{eqnarray}
is well defined on $V_{hk}(\Gamma)$ and is also continuous. Thus from Riesz representation theorem the equation (\ref{reconstruction}) admits a unique solution $\partial^{hk}_nZ_{hk}$ in $V_{hk}(\Gamma)$. For an analogous reconstruction of discrete normal derivatives for elliptic Dirichlet boundary control problems we refer to \cite{Casas.E;Raymond.JP2006}. With the help of (\ref{reconstruction}) it is not difficult to show that
\begin{eqnarray}
0&\leq& J'_{hk}(U_{hk})(v_{hk}-U_{hk})\nonumber\\
&=&\alpha\int_{\Sigma}U_{hk}(v_{hk}-U_{hk})dsdt+\int_{\Omega_T}(Y_{hk}-y_{d})(Y_{hk}(v_{hk})-Y_{hk})dxdt\nonumber\\
&=&\alpha\int_{\Sigma}U_{hk}(v_{hk}-U_{hk})dsdt+A(Y_{hk}(v_{hk})-Y_{hk},Z_{hk})
-\int_{\Sigma}\partial^{hk}_nZ_{hk}(Y_{hk}(v_{hk})-Y_{hk})dsdt\nonumber\\
&=&\alpha\int_{\Sigma}U_{hk}(v_{hk}-U_{hk})dsdt-\int_{\Sigma}\partial^{hk}_nZ_{hk}\cdot Q(v_{hk}-U_{hk})dsdt\nonumber\\
&=&\int_{\Sigma}(\alpha U_{hk}-\partial^{hk}_nZ_{hk})(v_{hk}-U_{hk})dsdt\nonumber
\end{eqnarray}
for $v_{hk}\in U_{ad}^{hk}$, which in turn implies
\begin{eqnarray}
U_{hk}=P_{U_{ad}^{hk}}({1\over\alpha}\partial^{hk}_nZ_{hk}),\label{projection_hk}
\end{eqnarray}
where $P_{U_{ad}^{hk}}:U\rightarrow U_{ad}^{hk}$ denotes the orthogonal projection in $U$ onto $U_{ad}^{hk}$.


\section{Error estimates for the optimal control problems} \setcounter{equation}{0}

As a preliminary result we first estimate the error introduced by the discretization of the state equation, i.e., the error between the solutions of problems (\ref{weak}) and (\ref{full_discrete}).

\begin{Theorem}\label{Thm:state_0}
Suppose that $f\in L^2(L^2(\Omega))$, $u\in L^2(L^2(\Gamma))$, and $y_0\in L^2(\Omega)$. Let $y\in L^2(L^2(\Omega))$ and $Y_{hk}(u)\in V_{hk}$ with $Y_{hk}(u)|_{\Sigma}=Q(u)$ be the solutions of problems (\ref{weak}) and (\ref{full_discrete}), respectively. Then for $k=O(h^2)$ we have
\begin{eqnarray}
\|y-Y_{hk}(u)\|_{L^2(L^2(\Omega))}\leq C(h^{\frac{1}{2}}+k^{\frac{1}{4}})\big(\|f\|_{L^2(L^2(\Omega))}+\|y_0\|_{0,\Omega}+\|u\|_{L^2(L^2(\Gamma))}\big).\label{parabolic_estimate_0}
\end{eqnarray}
\end{Theorem}
\begin{proof}
We first note that according to \cite{French.DA;King.JT1993} $y\in L^2(H^{1\over 2}(\Omega))$ holds. In view of the linearity of the problem it is sufficient to consider the problems with either $f\equiv 0$, $y_0\equiv 0$ or $u\equiv 0$.

Let us first assume that $f\equiv 0$, $y_0\equiv 0$. Let $g=y- Y_{hk}(u)$ in (\ref{auxiliary}). Since $w(T)=0$, we from (\ref{weak}) and (\ref{full_discrete}) deduce that
\begin{eqnarray}
&&\|y-Y_{hk}(u)\|_{L^2(L^2(\Omega))}^2\nonumber\\
&=&\int_0^T\int_{\Omega}(-w_t-\Delta w)(y- Y_{hk}(u))dxdt\label{E}\\
&=&\int_0^T\int_{\Omega}(-w_ty-\Delta wy)dxdt+\int_0^T\int_{\Omega}(w_tY_{hk}(u)-\nabla w\nabla Y_{hk}(u))dxdt+\int_0^T\int_{\Gamma} Q(u)\frac{\partial w}{\partial n}dsdt\nonumber\\
&=&\big (\int_0^T\int_{\Gamma}Q(u)\frac{\partial w}{\partial n}-\int_0^T\int_{\Gamma}u\frac{\partial w}{\partial n}\big )dsdt
+\int_0^T\int_{\Omega}\big(w_tY_{hk}(u)-\nabla w\nabla Y_{hk}(u)\big)dxdt\nonumber\\
&:=&E_1+E_2.\nonumber
\end{eqnarray}
We treat $E_1$ by exploiting the properties of $P_k$ and $Q_h$:
\begin{eqnarray}
E_1&=&\int_0^T\int_{\Gamma}( Q(u)-u)\frac{\partial w}{\partial n}dsdt\nonumber\\
&=&\int_0^T\big\langle(P_k-I)u,\frac{\partial w}{\partial n}\big\rangle dt+\int_0^T\big\langle( Q_h-I)P_ku, \frac{\partial w}{\partial n}\big\rangle dt\nonumber\\
&=&\int_0^T\big\langle(P_k-I)u,\frac{\partial}{\partial n}(I-P_k)w\big\rangle dt+\int_0^T\big\langle( Q_h-I)P_ku,(I-Q_h)\frac{\partial w}{\partial n}\big\rangle dt.\nonumber
\end{eqnarray}
From Young's inequality, the trace inequality and a norm interpolation inequality we derive (see, e.g. \cite{French.DA;King.JT1993})
\begin{eqnarray}
\|\frac{\partial}{\partial n}(I-P_k)w\|_{L^2(L^2(\Gamma))}^2&\leq& \frac{C}{\epsilon}\|(I-P_k)w\|_{L^2(H^2(\Omega))}^2+{\epsilon}\|(I-P_k)w\|_{L^2(H^1(\Omega))}^2\nonumber\\
&\leq&\frac{C}{\epsilon}\|(I-P_k)w\|_{L^2(H^2(\Omega))}^2+{\epsilon}\|(I-P_k)w\|_{L^2(H^2(\Omega))}\|(I-P_k)w\|_{L^2(L^2(\Omega))}\nonumber\\
&\leq&\frac{2C}{\epsilon}\|(I-P_k)w\|_{L^2(H^2(\Omega))}^2+{\epsilon^3}\|(I-P_k)w\|_{L^2(L^2(\Omega))}^2.\nonumber
\end{eqnarray}
Setting $\epsilon = k^{-\frac{1}{2}}$ and using the approximation property (\ref{Pk_error}) of $P_k$ gives
\begin{eqnarray}
\|\frac{\partial}{\partial n}(I-P_k)w\|_{L^2(L^2(\Gamma))}^2\leq Ck^{1/2}\big(\|w\|_{L^2(H^2(\Omega))}^2+\|w_t\|_{L^2(L^2(\Omega))}^2\big).\nonumber
\end{eqnarray}
We also have
\begin{eqnarray}
\|(I-Q_h)\frac{\partial w}{\partial n}\|_{L^2(L^2(\Gamma))}\leq Ch^{1/2}\|w\|_{L^2(H^2(\Omega))}.\nonumber
\end{eqnarray}
Using the Cauchy-Schwartz inequality and stability results for $Q_h$ and $P_k$ we estimate
\begin{eqnarray}
|E_1|&\leq& Ck^{1/4}\|u\|_{L^2(L^2(\Gamma))}\big(\|w\|_{L^2(H^2(\Omega))}+\|w_t\|_{L^2(L^2(\Omega))}\big)\nonumber\\
&&+Ch^{1/2}\|P_ku\|_{L^2(L^2(\Gamma))}\|w\|_{L^2(H^2(\Omega))}\nonumber\\
&\leq& C(h^{1/2}+k^{1/4})\|u\|_{L^2(L^2(\Gamma))}\|g\|_{L^2(L^2(\Omega))}.\label{E_1}
\end{eqnarray}
Next we estimate $E_2$. Considering (\ref{full_discrete}) and $w^N=w(\cdot,T)=0$ we calculate
\begin{eqnarray}
E_2&=&\int_0^T\int_{\Omega}\big(w_tY_{hk}(u)-\nabla w\nabla Y_{hk}(u)\big)dxdt\nonumber\\
&=&\sum\limits_{i=1}^N\big(Y_{hk}^i(u),w^{i}-w^{i-1}\big)-k(\nabla Y_{hk}^i(u),\nabla P_k^iw)\nonumber\\
&=&-\sum\limits_{i=1}^N\big(Y_{hk}^i(u)- Y_{hk}^{i-1}(u),w^{i-1}\big)+k(\nabla Y_{hk}^i(u),\nabla P_k^iw)\nonumber\\
&=&-\sum\limits_{i=1}^N(Y_{hk}^i(u)- Y_{hk}^{i-1}(u),w^{i-1}- R_hP_k^iw)+k\big(\nabla Y_{hk}^i(u),\nabla (P_k^iw-R_hP_k^iw)\big).\label{E2}
\end{eqnarray}
By the Cauchy-Schwartz inequality we have
\begin{eqnarray}
|E_2|\leq F_1\cdot F_2,\nonumber
\end{eqnarray}
where
\begin{eqnarray}
F_1=\Big(\sum\limits_{i=1}^N\big(\| Y_{hk}^i(u)- Y_{hk}^{i-1}(u)\|_{0,\Omega}^2+k(\nabla   Y_{hk}^i(u),\nabla  Y_{hk}^i(u))\big)\Big)^{{1\over 2}}\nonumber
\end{eqnarray}
and
\begin{eqnarray}
F_2=\Big(\sum\limits_{i=1}^N\big(\|w^{i-1}- R_hP_k^iw\|_{0,\Omega}^2+k(\nabla(I-  R_h)P_k^iw,\nabla(I- R_h)P_k^iw)\big)\Big)^{{1\over 2}}.\nonumber
\end{eqnarray}
In view of the stability result (\ref{stability}) of Lemma \ref{stab} we have
\begin{eqnarray}
|F_1|\leq Ch^{-{1\over 2}}\|u\|_{L^2(L^2(\Gamma))}.\label{F_1}
\end{eqnarray}
It remains to estimate $F_2$. To begin with we note that
\begin{eqnarray}
\|w^{i-1}-R_hP_k^iw\|_{0,\Omega}&\leq & \|w^{i-1}- P_k^iw\|_{0,\Omega}+\|(I-  R_h)P_k^iw\|_{0,\Omega}\nonumber\\
&\leq&\|w^{i-1}-P_k^iw\|_{0,\Omega}+Ch^2\|P_k^iw\|_{2,\Omega},\label{F2_1}
\end{eqnarray}
and
\begin{eqnarray}
(\nabla(I- R_h)P_k^iw,\nabla(I-R_h)P_k^iw)\leq Ch^2\|P_k^iw\|_{2,\Omega}^2.\label{F2_2}
\end{eqnarray}
It is straightforward to show that
\begin{eqnarray}
\|w^{i-1}-P_k^iw\|_{0,\Omega}\leq k^{1/2}\|w_t\|_{L^2(I_i,L^2(\Omega))}\label{F2_3}
\end{eqnarray}
and
\begin{eqnarray}
\|P_k^iw\|_{2,\Omega}\leq k^{-1/2}\|w\|_{L^2(I_i,H^2(\Omega))}.\label{F2_4}
\end{eqnarray}
Combining (\ref{F2_1})-(\ref{F2_4}) we get
\begin{eqnarray}
F_2&\leq& C\Big(\sum\limits_{i=1}^N((h^4+kh^2)\|P_k^iw\|_{2,\Omega}^2+k\|w_t\|_{L^2(I_i,L^2(\Omega))}^2)\Big)^{{1\over 2}}\nonumber\\
&\leq&C(h+k^{{1\over 2}})\Big(\sum\limits_{i=1}^N(\|w\|_{L^2(I_i,H^2(\Omega))}^2+\|w_t\|_{L^2(I_i,L^2(\Omega))}^2)\Big)^{{1\over 2}}\nonumber\\
&\leq&C(h+k^{{1\over 2}})\big(\|w\|_{L^2(H^2(\Omega))}+\|w_t\|_{L^2(L^2(\Omega))}\big).
\end{eqnarray}
Using the stability estimate of Lemma \ref{stab} we conclude
\begin{eqnarray}
|E_2|\leq Ch^{-{1\over 2}}(h+k^{{1\over 2}})\big(\|w\|_{L^2(H^2(\Omega))}+\|w_t\|_{L^2(L^2(\Omega))}\big)\|u\|_{L^2(L^2(\Gamma))}.\label{E_3}
\end{eqnarray}
From the estimates (\ref{E})-(\ref{E_3}) we conclude the desired result in the case $f\equiv 0$, $y_0\equiv 0$.
\begin{eqnarray}
\|y-Y_{hk}(u)\|_{L^2(L^2(\Omega))}\leq C(h^{{1\over 2}}+k^{{1\over 4}})\|u\|_{L^2(L^2(\Gamma))}.
\end{eqnarray}

If $u\equiv0$, $f\in L^2(L^2(\Omega))$ and $y_0\in L^2(\Omega)$ we have $y\in L^2(H_0^1(\Omega))\cap H^1(H^{-1}(\Omega))$ (see, e.g., \cite{Lions.JL;Magenes.E1972}). Then it is straightforward to prove that
\begin{eqnarray}
\|y- Y_{hk}(u)\|_{L^2(L^2(\Omega))}\leq C(h+k^{1\over 2})(\|f\|_{L^2(L^2(\Omega))}+\|y_0\|_{0,\Omega}).
\end{eqnarray}
Combining both cases completes the proof.
\end{proof}

Now we are in a position to derive our main result of this section: the a priori error estimates for optimal control problems. At first we consider the fully discrete case, i.e., $U_{ad}^{hk}=U_{ad}\cap V_{hk}(\Gamma)$.
\begin{Theorem}\label{Thm:fully_discrete}
Let $(y,u,z)\in {L^2(L^2(\Omega))}\times {L^2(L^2(\Gamma))}\times {L^2(H^2(\Omega))}\cap H^1(L^2(\Omega))$ and $(Y_{hk},U_{hk},Z_{hk})\in V_{hk}\times U_{ad}^{hk}\times V_{hk}^0$ be the solutions of problem (\ref{OPT_weak})-(\ref{optimality1}) and (\ref{state_full})-(\ref{adjoint_control_full}), respectively. Then  we have the a priori error estimate
\begin{eqnarray}
\|u-U_{hk}\|_{L^2(L^2(\Gamma))}+\|y-Y_{hk}\|_{L^2(L^2(\Omega))}\leq C(h^{1\over 2}+k^{{1\over 4}}),\label{main_result}
\end{eqnarray}
with a constant $C>0$ independent of $h$ and $k$. 
\end{Theorem}

\begin{proof}
Let us recall the continuous and discrete optimality conditions
\begin{eqnarray}
\int_0^T\int_{\Omega}(y-y_d)(y(v)-y)dxdt+\alpha\int_0^T\int_{\Gamma} u(v-u)dsdt\geq 0,\ \ \forall\ v\in U_{ad},\label{inequality_1}
\end{eqnarray}
and
\begin{equation}
\int_{0}^{T}\int_{\Omega}(Y_{hk}- y_{d})(Y_{hk}(v_{hk})- Y_{hk})dxdt+\alpha\int_{0}^{T}\int_{\Gamma} U_{hk}(v_{hk}- U_{hk})dsdt\geq 0,\ \forall\ v_{hk}\in U_{ad}^{hk}.\label{inequality_2}
\end{equation}
Setting $v=U_{hk}\in U_{ad}$ and $v_{hk}=Q(u)\in U_{ad}^{hk}$ we have
\begin{eqnarray}
&&\alpha\|u- U_{hk}\|_{L^2(L^2(\Gamma))}^2=\alpha\int_0^T\int_{\Gamma}(u- U_{hk})^2dsdt\nonumber\\
&=&\alpha\int_0^T\int_{\Gamma}u(u- U_{hk})dsdt-\alpha\int_0^T\int_{\Gamma} U_{hk}(u- U_{hk})dsdt\nonumber\\
&\leq&\int_0^T\int_{\Omega}(y-y_d)(y(U_{hk})-y)dxdt-\alpha\int_0^T\int_{\Gamma} U_{hk}(u- Q(u))dsdt\nonumber\\
&&-\alpha\int_0^T\int_{\Gamma}U_{hk}(Q(u)-U_{hk})dsdt\nonumber\\
&\leq&\int_0^T\int_{\Omega}(y-y_d)(y(U_{hk})-y)dxdt+\int_0^T\int_{\Omega}( Y_{hk}-y_d)( Y_{hk}(Qu)-Y_{hk})dxdt\nonumber\\
&&-\alpha\int_0^T\int_{\Gamma}U_{hk}(u- Q(u))dsdt,\label{4.2.1}
\end{eqnarray}
where $y(U_{hk})\in L^2(L^2(\Omega))$ with $y(U_{hk})|_{\Sigma}= U_{hk}$ solves
\begin{eqnarray}
\int_{\Omega_T}y(U_{hk})(-v_t-\Delta v)dxdt=-\int_{\Sigma}U_{hk}\partial_{{n}}vdsdt+\int_{\Omega_T}fvdxdt+\int_\Omega y_0v(\cdot,0)dx\nonumber\\
\ \ \forall\ v\in L^2(H^2(\Omega)\cap H_0^1(\Omega))\cap H^1(L^2(\Omega))\label{auxiliary1}
\end{eqnarray}
with $v(\cdot,T)=0$, and $Y_{hk}(Qu)\in V_{hk}$ solves
\begin{equation}\label{auxiliary2}
\left\{\begin{aligned}
A(Y_{hk}(Qu),\Phi)=(f,\Phi)_{\Omega_T}+(y_0,\Phi_+^0),\ \ &\forall\ \Phi\in V_{hk}^0,\\
Y_{hk}(Qu)=Q(u)\ \ &\mbox{on}\ \Sigma.
\end{aligned}\right.
\end{equation}
With Young's inequality we deduce
\begin{eqnarray}
&&\int_0^T\int_{\Omega}(y-y_d)(y(U_{hk})-y)dxdt+\int_0^T\int_{\Omega}( Y_{hk}-y_d)( Y_{hk}(Qu)- Y_{hk})dxdt\nonumber\\
&=&(y-y_d,y(U_{hk})-y)_{\Omega_T}+(Y_{hk}-y_d,Y_{hk}(Qu)- Y_{hk})_{\Omega_T}\nonumber\\
&=&(y-y_d,y(U_{hk})-y)_{\Omega_T}+( Y_{hk}-y,y-Y_{hk})_{\Omega_T}\nonumber\\
&&+(Y_{hk}-y, Y_{hk}(Qu)-y)_{\Omega_T}+(y-y_d, Y_{hk}(Qu)- Y_{hk})_{\Omega_T}\nonumber\\
&=&-\|y-Y_{hk}\|_{L^2(L^2(\Omega))}^2+(Y_{hk}-y, Y_{hk}(Qu)-y)_{\Omega_T}\nonumber\\
&&+\big(y-y_d,y( U_{hk})-y-( Y_{hk}-Y_{hk}(Qu))\big)_{\Omega_T}\nonumber\\
&\leq&-\|y-Y_{hk}\|_{L^2(L^2(\Omega))}^2+\big(y-y_d,y(U_{hk})-y-(Y_{hk}- Y_{hk}(Qu))\big)_{\Omega_T}\nonumber\\
&&+\sigma\|y-Y_{hk}\|_{L^2(L^2(\Omega))}^2+C(\sigma)\| Y_{hk}(Qu)-y\|_{L^2(L^2(\Omega))}^2.\label{4.2.2}
\end{eqnarray}
Taking $\sigma>0$ small enough, we from (\ref{4.2.1})-(\ref{4.2.2}) obtain
\begin{eqnarray}
&&\alpha\|u-U_{hk}\|_{L^2(L^2(\Gamma))}^2+\|y- Y_{hk}\|_{L^2(L^2(\Omega))}^2\nonumber\\
&\leq &-\alpha\int_0^T\int_{\Gamma} U_{hk}(u- Q(u))dsdt+C\| Y_{hk}(Qu)-y\|_{L^2(L^2(\Omega))}^2\nonumber\\
&&+(y-y_d,y(U_{hk})-y-( Y_{hk}- Y_{hk}(Qu)))_{\Omega_T}\nonumber\\
&:=&I_1+I_2+I_3.
\end{eqnarray}
Note that from the standard error estimates for the $L^2$-projection and the regularity of $u$ we have
\begin{eqnarray}
\|Q(u)-u\|_{L^2(L^2(\Gamma))}^2\leq C(h+k^{1\over 2})\big(\|u\|_{L^2(H^{\frac{1}{2}}(\Gamma))}^2+\|u\|_{H^{\frac{1}{4}}(L^2(\Gamma))}^2\big).
\end{eqnarray}
Thus we are led to
\begin{eqnarray}
|I_1|&=&\big|-\alpha\int_0^T\int_{\Gamma} U_{hk}(u-Q(u))dsdt\big|\nonumber\\
&=&\big|\alpha\int_0^T\int_{\Gamma}u( Q(u)-u)dsdt+\alpha\int_0^T\int_{\Gamma}(U_{hk}-u)( Q(u)-u)dsdt\big|\nonumber\\
&=&\big|\alpha\int_0^T\int_{\Gamma}(u- Q(u))( Q(u)-u)dsdt +\alpha\int_0^T\int_{\Gamma}(U_{hk}-u)( Q(u)-u)dsdt\big|\nonumber\\
&\leq&\sigma\|u- U_{hk}\|^2_{L^2(L^2(\Gamma))}+C\|u- Q(u)\|^2_{L^2(L^2(\Gamma))}\nonumber\\
&\leq&\sigma\|u- U_{hk}\|^2_{L^2(L^2(\Gamma))}+C(h+k^{1\over 2})\big(\|u\|_{L^2(H^{\frac{1}{2}}(\Gamma))}^2+\|u\|_{H^{\frac{1}{4}}(L^2(\Gamma))}^2\big).
\end{eqnarray}
Since $Y_{hk}(Qu)$ is the fully discrete finite element approximation of $y$, the error estimate (\ref{parabolic_estimate_0}) of Theorem \ref{Thm:state_0} gives
\begin{eqnarray}
I_2=\|Y_{hk}(Qu)-y\|_{L^2(L^2(\Omega))}^2\leq C(h+k^{1\over 2})\|u\|_{L^2(L^2(\Gamma))}^2.
\end{eqnarray}
Then it remains to estimate $I_3$. From (\ref{weak}), (\ref{OPT_adjoint}), (\ref{auxiliary1}) and (\ref{auxiliary2}) we have
\begin{eqnarray}
I_3&=&(y-y_d,y(U_{hk})-y-(Y_{hk}- Y_{hk}(Qu)))_{\Omega_T}\nonumber\\
&=&\int_0^T\int_{\Omega}(y(U_{hk})-y)(-\frac{\partial z}{\partial t}-\Delta z)dxdt-\int_0^T\int_{\Omega}(Y_{hk}- Y_{hk}(Qu))(-\frac{\partial z}{\partial t}-\Delta z)dxdt\nonumber\\
&=&\int_0^T\int_{\Omega}\big(-(y(U_{hk})-y)z_t-(y( U_{hk})-y)\Delta z\big)dxdt\nonumber\\
&&+\int_0^T\int_{\Omega}\big(z_t(Y_{hk}- Y_{hk}(Qu))-\nabla(Y_{hk}-Y_{hk}(Qu))\nabla z\big )dxdt+\int_0^T\int_{\Gamma}(U_{hk}-Q(u))\partial_{{ n}}z\nonumber\\
&=&-\int_0^T\int_{\Gamma}( U_{hk}-u)\partial_{{ n}}zdsdt+\int_0^T\int_{\Gamma}(U_{hk}-Q(u))\partial_{{n}}zdsdt\nonumber\\
&&+\int_0^T\int_{\Omega}\big(z_t(Y_{hk}- Y_{hk}(Qu))-\nabla(Y_{hk}-Y_{hk}(Qu))\nabla z\big)dxdt\nonumber\\
&=&H_1+H_2+H_3.
\end{eqnarray}
Note that
\begin{eqnarray}
H_1+H_2&=&-\int_0^T\int_{\Gamma}(U_{hk}-u)\partial_{{ n}}zdsdt+\int_0^T\int_{\Gamma}(U_{hk}-Q(u))\partial_{{ n}}zdsdt\nonumber\\
&=&-\int_0^T\int_{\Gamma}\big(U_{hk}-u- Q(U_{hk}-u)\big)\partial_{{ n}}zdsdt\nonumber\\
&=&-\int_0^T\int_{\Gamma}\big(U_{hk}-u- Q(U_{hk}-u)\big)\big(\partial_{{n}}z- Q(\partial_{{n}}z)\big)dsdt.\nonumber
\end{eqnarray}
It is straightforward to estimate
\begin{eqnarray}
|H_1+H_2|&\leq& C(h+k^{1\over 2})\big(\|z\|_{L^2(H^2(\Omega))}^2+\|z_t\|_{L^2(L^2(\Omega))}^2\big)+\delta\|u- U_{hk}\|^2_{L^2(L^2(\Gamma))}.
\end{eqnarray}
Define $E_{hk}:=Y_{hk}- Y_{hk}(Qu)$. Using the proof technique of Theorem \ref{Thm:state_0} we from (\ref{state_full}), (\ref{auxiliary2}) obtain
\begin{eqnarray}
H_3&=&\int_0^T\int_{\Omega}\big(z_tE_{hk}-\nabla E_{hk}\nabla z\big )dxdt\nonumber\\
&=&\sum\limits_{i=1}^N\big((E_{hk}^i,z^i-z^{i-1})-k(\nabla E_{hk}^i,\nabla P_k^iz)\big),\nonumber\\
&=&-\sum\limits_{i=1}^N\big((E_{hk}^i-E_{hk}^{i-1},z^{i-1})+k(\nabla E_{hk}^i,\nabla P_k^iz)\big),\nonumber\\
&=&-\sum\limits_{i=1}^N(E_{hk}^i-E_{hk}^{i-1},z^{i-1}-R_hP_k^iz)+k(\nabla E_{hk}^i,\nabla (P_k^iz-R_hP_k^iz)).\nonumber
\end{eqnarray}
With the help of projection error estimate and proceeding as in the estimate of (\ref{E2}) we have
\begin{eqnarray}
|H_3|&\leq& C(h+k^{1\over 2})\big(\|z\|_{L^2(H^2(\Omega))}+\|z\|_{H^1(L^2(\Omega))}\big)\big(\sum\limits_{i=1}^{N}\|E_{hk}^i- E_{hk}^{i-1}\|^2_{0,\Omega}+k\|E_{hk}^i\|^2_{1,\Omega}\big)^{1\over 2}.\nonumber
\end{eqnarray}
From Lemma \ref{stab} we conclude
\begin{eqnarray}
|H_3|&\leq&C(h+k^{1\over 2})\big(\|z\|_{L^2(H^2(\Omega))}+\|z_t\|_{L^2(L^2(\Omega))}\big)h^{-{1\over 2}}\|Q(u)-U_{hk}\|_{L^2(L^2(\Gamma))}\nonumber\\
&\leq&C(h^{{1\over 2}}+k^{1\over 4})\big(\|z\|_{L^2(H^2(\Omega))}+\|z_t\|_{L^2(L^2(\Omega))}\big)\|u- U_{hk}\|_{L^2(L^2(\Gamma))},
\end{eqnarray}
where we have used the fact that $k=O(h^2)$. Combining above results and using the Cauchy-Schwartz inequality and Young's inequality completes the proof.
\end{proof}

If we use variational discretization concept introduced in \cite{Hinze.M2005}, i.e., $U_{ad}^{hk}=U_{ad}$, we can prove the following error estimates in a similar way.
\begin{Theorem}\label{Thm:variational_discrete}
Let $(y,u,z)\in {L^2(L^2(\Omega))}\times {L^2(L^2(\Gamma))}\times {L^2(H^2(\Omega))}\cap H^1(L^2(\Omega))$ and $(Y_{hk},U_{hk},Z_{hk})\in V_{hk}\times U_{ad}\times V_{hk}^0$ be the solutions of problem (\ref{OPT_weak})-(\ref{optimality1}) and (\ref{state_full})-(\ref{adjoint_control_full}), respectively. Then we have following a priori error estimate
\begin{eqnarray}
\|u-U_{hk}\|_{L^2(L^2(\Gamma))}+\|y-Y_{hk}\|_{L^2(L^2(\Omega))}\leq C(h^{1\over 2}+k^{{1\over 4}})\label{main_result1}
\end{eqnarray}
with a constant $C>0$ independent of $h$ and $k$.
\end{Theorem}
\begin{proof}
In the proof of Theorem \ref{Thm:fully_discrete} it suffices to set $v=U_{hk}$ in (\ref{inequality_1}) and $v_{hk}=u$ in (\ref{inequality_2}) and add the corresponding inequalities. This directly gives
\begin{eqnarray}
\alpha\|u-U_{hk}\|_{L^2(L^2(\Gamma))}^2&\leq& (y-y_d,y(U_{hk})-y)_{\Omega_T}+( Y_{hk}-y_d,Y_{hk}(u)-Y_{hk})_{\Omega_T}\nonumber\\
&\leq&-\|y-Y_{hk}\|_{L^2(L^2(\Omega))}^2+(y-Y_{hk},y-Y_{hk}(u))\nonumber\\
&&+(y-y_d,y(U_{hk})-y-( Y_{hk}- Y_{hk}(u)))_{\Omega_T}.\nonumber
\end{eqnarray}
Thus
\begin{eqnarray}
&&\alpha\|u-U_{hk}\|_{L^2(L^2(\Gamma))}^2+\|y-Y_{hk}\|_{L^2(L^2(\Omega))}^2\nonumber\\
&\leq& C\|y-Y_{hk}(u)\|_{L^2(L^2(\Omega))}^2+(y-y_d,y(U_{hk})-y-( Y_{hk}- Y_{hk}(u)))_{\Omega_T}.\nonumber
\end{eqnarray}
The rest of proof is along the lines of the estimation of the term $I_3$ in the proof of Theorem \ref{Thm:fully_discrete}.

\end{proof}

\begin{Remark}
The error estimates we obtained in Theorem \ref{Thm:fully_discrete} and \ref{Thm:variational_discrete} reflect the worst cases we can expect for parabolic Dirichlet boundary control problems defined on convex polygonal domains. Since the regularity of parabolic equations depends on the maximum interior angle of the domain, the state admits the improved regularity $y\in L^s(0,T;W^{1,p}(\Omega))$ for $s,p \geq 2$ depending on the data and the domain. Moreover, for problems defined on curved domains with smooth boundary, we have higher regularity for the optimal control problems as stated in Theorem \ref{reg_curve}. Improved regularity leads to better approximation properties of the state and thus to better convergence rates, as is reported in our numerical results. For the elliptic case with polygonal boundaries we refer to \cite{Casas.E;Raymond.JP2006} where an approximation order for the controls of $O(h^{1-\frac{1}{p}})$ is derived for some $2<p\leq p_*$ with $p_*$ depending on the data and the maximum interior angle of the domain. For the error estimates of elliptic Dirichlet boundary control problems defined on curved domains we refer to \cite{Casas.E;Sokolowski.J2010}, \cite{Deckelnick.K;Gunther.A;Hinze.M2009SICON} and \cite{Gong.W;Yan.NN2011SICON} for more details.
\end{Remark}


\section{Numerical experiments} \setcounter{equation}{0}
In this section we will carry out some numerical experiments to support our theoretical findings. We consider the optimal control problems (\ref{OPT})-(\ref{OPT_state}) of tracking type with control set $U_{ad}$ defined as follows
\begin{eqnarray}
U_{ad}:=\big\{u\in L^2(0,T;L^2(\Gamma)):\ 0\leq u(x,t)\leq 1,\ \ \mbox{a.a.}\ (t,x)\in [0,T]\times\Gamma \big\}.\nonumber
\end{eqnarray}
Although we do not consider problems defined on curved domains in our numerical analysis, we include some numerical examples on both polygonal and curved domains using full discretization and variational discretisation of the control. For the numerical approximations of Dirichlet boundary control problems defined on curved domains we refer to \cite{Casas.E;Sokolowski.J2010}, \cite{Deckelnick.K;Gunther.A;Hinze.M2009SICON} and \cite{Gong.W;Yan.NN2011SICON}.
\begin{Example}\label{Ex:6.1}
The first example
is a unconstrained problem defined on the unit square $\Omega=[0,1]\times
[0,1]$, $T=1$. The data is chosen as
\begin{eqnarray}
f&=&-\frac{4.0}{\alpha}\sin(\pi t) - \frac{\pi}{\alpha}(x_1(1-x_1)+x_2(1-x_2))\cos(\pi t),\nonumber \\
y_d&=&-(2+1.0/\alpha)(x_1(1-x_1)+x_2(1-x_2))\sin(\pi t) + \pi x_1x_2(1-x_1)(1-x_2)\cos(\pi t),\nonumber
\end{eqnarray}
with $\alpha =1$, so that the optimal solution is given by
\begin{eqnarray}
u=-\frac{1}{\alpha}(x_1(1-x_1)+x_2(1-x_2))\sin(\pi t),\nonumber\\
y=-\frac{1}{\alpha}(x_1(1-x_1)+x_2(1-x_2))\sin(\pi t),\nonumber\\
z=x_1x_2(1-x_1)(1-x_2)\sin(\pi t).\nonumber
\end{eqnarray}

\end{Example}

At first we consider the error with respect to spatial mesh size $h$. We fix the time step to $k= \frac{1}{4096}$ and present the errors of optimal
control $u$, state $y$ and adjoint state $z$ in Table ~\ref{table:6.1}. Then we consider the convergence order of error with respect to time step size $k$. We fixed the space mesh with $DOF= 22785$ and present the errors of optimal control $u$, state $y$ and adjoint state $z$ in Table ~\ref{table:6.2}. We observe an order of convergence $\frac{3}{2}$ for the control and  order $2$ for the state and adjoint state for spatial discretization, and order $1$ convergence for both of them for the time discretization. This is the best result we can expect for linear finite elements and dG(0) approximations.

\begin{table}
\centering
\caption{ Error of control $u$, state $y$ and adjoint state $z$ for Example \ref{Ex:6.1} with fixed time step $N=4096$.}\label{table:6.1}
\begin{tabular}{|c|c|c|c|c|c|c|}
\hline
$Dof$&$\|u-U_{hk}\|_{{L^2(L^2(\Gamma))}}$&order&$\|y-Y_{hk}\|_{{L^2(L^2(\Omega))}}$&order&$\|z-Z_{hk}\|_{{L^2(L^2(\Omega))}}$&order\\
\hline
31&0.026183462176&$\setminus$&0.007511873255&$\setminus$&0.001757302801&$\setminus$ \\
\hline
105&0.010011019039&1.5762&0.001941841224&2.2178&0.000449449353&2.2353\\
\hline
385&0.003690706908&1.5360&0.000497530721&2.0961&0.000112846733&2.1273\\
\hline
1473&0.001333520373&1.5173&0.000124966909&2.0593&0.000031576461&1.8984 \\
\hline
5761&0.000477713130&1.5054&0.000034179205&1.9012&0.000017536745&0.8625\\
\hline
 \end{tabular}
\end{table}

\begin{table}
\centering
\caption{ Error of control $u$, state $y$ and adjoint state $z$ for Example \ref{Ex:6.1} with fixed mesh $Dof=22785$.}\label{table:6.2}
\begin{tabular}{|c|c|c|c|c|c|c|}
\hline
$N$&$\|u-U_{hk}\|_{{L^2(L^2(\Gamma))}}$&order&$\|y-Y_{hk}\|_{{L^2(L^2(\Omega))}}$&order&$\|z-Z_{hk}\|_{{L^2(L^2(\Omega))}}$&order\\
\hline
2&0.030398571028&$\setminus$&0.025518889725&$\setminus$&0.030398678795&$\setminus$ \\
\hline
4&0.018595263504&0.7091&0.015608671639&0.7092&0.016205193500&0.9076\\
\hline
8&0.011269239439&0.7225&0.008709737707&0.8416&0.008295451927&0.9661\\
\hline
16&0.006647997579&0.7614&0.004633730272&0.9105&0.004196415068&0.9832 \\
\hline
32&0.003745029192&0.8279&0.002392405119&0.9537&0.002112281353&0.9904\\
\hline
64&0.002021680314&0.8894&0.001213660252&0.9791&0.001060211691&0.9944 \\
\hline
128&0.001064326895&0.9256&0.000610081508&0.9923&0.000531194879&0.9970 \\
\hline
256&0.000563475024&0.9175&0.000305197979&0.9993&0.000265839733&0.9987 \\
\hline
512&0.000320272343&0.8151&0.000152181075&1.0040&0.000132937121&0.9998 \\
\hline
 \end{tabular}
\end{table}

\begin{Example}\label{Ex:6.2b}
The second example is an unconstrained problem defined in a polygonal
domain with maxminum interior angle $\omega_{max}={5\over 6}\pi$(see
\cite{May.S;Rannacher.R;Vexler.B2008}), the optimal solution has
only reduced regularity. The data is chosen as
\begin{eqnarray}
\nonumber
y_d=
\left\{ \begin{aligned}
&t^2g(x)\ \ &&0\leq t<0.5, \\
&-t^2g(x)\ \ \ &&0.5\leq t \leq 1
\end{aligned} \right.
\end{eqnarray}
with $f=1$, $g=\frac{1}{(x_1^2+x_2^2)^{\frac{1}{3}}}$.
\end{Example}

There is no exact solution for this example. We take the solution with $k=\frac{1}{4096}$ and $Dof=158561$ in the spatial discretization as reference solution. Similarly as in the previous example, we investigate the convergence order with respect to the spatial and time discretization separately. We in Table \ref{table:6.2a} can observe nearly $O(h^{1\over 2})$-order convergence for the spatial discretization of the control. This is in agreement with our theoretical results. The convergence order for the time discretization reported in Table \ref{table:6.2b} is higher than $O(k^{1\over 4})$ which is caused by the higher regularity of the state w.r.t time.
\begin{table}
\centering
\caption{ Error of control $u$, state $y$ and adjoint state $z$ for Example \ref{Ex:6.2b} with fixed time step $N=4096$.}\label{table:6.2a}
\begin{tabular}{|c|c|c|c|c|c|c|}
\hline
$Dof$&$\|u-U_{hk}\|_{{L^2(L^2(\Gamma))}}$&order&$\|y-Y_{hk}\|_{{L^2(L^2(\Omega))}}$&order&$\|z-Z_{hk}\|_{{L^2(L^2(\Omega))}}$&order\\
\hline
53&0.021578034106&$\setminus$&0.006632500684&$\setminus$&0.002058710969&$\setminus$ \\
\hline
182&0.013947673670&0.7074&0.002850064176&1.3693&0.000543409130&2.1593\\
\hline
671&0.010619549501&0.4179&0.001537497778&0.9460&0.000141300731&2.0647\\
\hline
2573&0.008212461016&0.3825&0.000866946011&0.8525&0.000036700709&2.0060\\
\hline
10073&0.006098931924&0.4360&0.000467699526&0.9044&0.000009388603&1.9978\\
\hline
average&$\setminus$&0.4859&$\setminus$&1.0180&$\setminus$&2.0570\\
\hline
 \end{tabular}
\end{table}

\begin{table}
\centering
\caption{ Error of control $u$, state $y$ and adjoint state $z$ for Example \ref{Ex:6.2b} with fixed mesh $Dof=158561$.}\label{table:6.2b}
\begin{tabular}{|c|c|c|c|c|c|c|}
\hline
$N$&$\|u-U_{hk}\|_{{L^2(L^2(\Gamma))}}$&order&$\|y-Y_{hk}\|_{{L^2(L^2(\Omega))}}$&order&$\|z-Z_{hk}\|_{{L^2(L^2(\Omega))}}$&order\\
\hline
4&0.167327134403&$\setminus$&0.088804542387&$\setminus$&0.020552016247&$\setminus$ \\
\hline
8&0.107460725354&0.6389&0.051579167598&0.7838&0.013212388713&0.6374\\
\hline
16&0.070415333055&0.6098&0.030045232355&0.7797&0.007759076126&0.7679\\
\hline
32&0.044502497910&0.6620&0.016906812048&0.8295&0.004258097400&0.8657\\
\hline
64&0.026671216305&0.7386&0.009227011579&0.8737&0.002231368045&0.9323 \\
\hline
average&$\setminus$&0.6623&$\setminus$&0.8167&$\setminus$&0.8008\\
\hline
 \end{tabular}
\end{table}

\begin{Example}\label{Ex:6.3}
This example is a control constrained problem defined in a smooth
domain (see \cite{Deckelnick.K;Gunther.A;Hinze.M2009SICON}). The domain is the unit circle
$\Omega=B(0,1)$ with center at zero and radius 1, $T=1$. The data is
presented in polar coordinates. We set
\begin{eqnarray}
f(r,\theta;t)=-6r\max(0,\cos{\theta}sin^3(\pi t))-{\pi\over 2} sin(\pi t)r^3\max(0,\cos^3{\theta}),\nonumber\\
y_d(r,\theta;t)=(7r^2\cos^2{\theta}+6r^2-6r)\cos{\theta}sin^3(\pi t)+y(r,\theta)-{\pi\over 2} sin(\pi t)r^3(r-1)\max(0,\cos^3{\theta}),\nonumber
\end{eqnarray}
so that the optimal solution is given by
\begin{eqnarray}
u(r,\theta;t)=\max(0,\cos^3{\theta}sin^3(\pi t)),\nonumber\\
y(r,\theta;t)=r^3\max(0,\cos^3{\theta}sin^3(\pi t)),\nonumber\\
z(r,\theta;t)=r^3(r-1)\cos^3{\theta}sin^3(\pi t).\nonumber
\end{eqnarray}
We set $\alpha=1$.
\end{Example}

First we consider the error with respect to spatial mesh size $h$. We fix the time step $k= \frac{1}{4096}$ and present the error of the optimal
control $u$, the state $y$ and the adjoint state $z$ in Table ~\ref{table:6.5} with full discretisation, and in Table ~\ref{table:6.51} with variational discretisation. We as expected observe that both approaches deliver similar results.
Then we consider the convergence order of the time error. We fix the space mesh with $DOF= 16641$ and present the error of the optimal control $u$, the state $y$ and the adjoint state $z$ in Table ~\ref{table:6.6}. We observe higher order convergence w.r.t. the spatial discretization for both the control $u$ and the state $y$.

\begin{table}
\centering
\caption{  Error of control $u$, state $y$ and adjoint state $z$ for Example \ref{Ex:6.3} with fixed time step $N=4096$ and full discretisation.}\label{table:6.5}
\begin{tabular}{|c|c|c|c|c|c|c|}
\hline
$Dof$&$\|u-U_{hk}\|_{{L^2(L^2(\Gamma))}}$&order&$\|y-Y_{hk}\|_{{L^2(L^2(\Omega))}}$&order&$\|z-Z_{hk}\|_{{L^2(L^2(\Omega))}}$&order\\
\hline
25&0.080438933409&$\setminus$&0.050499450897&$\setminus$&0.027537485412&$\setminus$ \\
\hline
81&0.052462945814&0.6166&0.018390173000&1.4573&0.009386608073&1.5527\\
\hline
289&0.025693482841&1.0299&0.005955737135&1.6266&0.002546961550&1.8818\\
\hline
1089&0.010836772478&1.2455&0.001775580269&1.7460&0.000654383533&1.9606 \\
\hline
4225&0.004214559039&1.3625&0.000499943184&1.8285&0.000180980668&1.8543\\
\hline
 \end{tabular}
\end{table}

\begin{table}
\centering
\caption{  Error of control $u$, state $y$ and adjoint state $z$ for Example \ref{Ex:6.3} with fixed time step $N=4096$ and variational discretisation.}\label{table:6.51}
\begin{tabular}{|c|c|c|c|c|c|c|}
\hline
$Dof$&$\|u-U_{hk}\|_{{L^2(L^2(\Gamma))}}$&order&$\|y-Y_{hk}\|_{{L^2(L^2(\Omega))}}$&order&$\|z-Z_{hk}\|_{{L^2(L^2(\Omega))}}$&order\\
\hline
25&0.080080823101&$\setminus$&0.050452860374&$\setminus$&0.027529915188&$\setminus$ \\
\hline
81&0.052889327427&0.5985&0.018972070528&1.4111&0.009398192721&1.5506\\
\hline
289&0.025433412568&1.0563&0.005947170731&1.6736&0.002550336513&1.8817\\
\hline
1089&0.010686498089&1.2509&0.001738641208&1.7742&0.000652324545&1.9670 \\
\hline
4225&0.004164564394&1.3596&0.000485052340&1.8417&0.000167939460&1.9576\\
\hline
 \end{tabular}
\end{table}

\begin{table}
\centering
\caption{ Error of control $u$, state $y$ and adjoint state $z$ for Example \ref{Ex:6.3} with fixed mesh $Dof=16641$.}\label{table:6.6}
\begin{tabular}{|c|c|c|c|c|c|c|}
\hline
$N$&$\|u-U_{hk}\|_{{L^2(L^2(\Gamma))}}$&order&$\|y-Y_{hk}\|_{{L^2(L^2(\Omega))}}$&order&$\|z-Z_{hk}\|_{{L^2(L^2(\Omega))}}$&order\\
\hline
2&0.018744907559&$\setminus$&0.027982932366&$\setminus$&0.071397322443&$\setminus$ \\
\hline
4&0.024808047492&-0.4043&0.020226462340&0.4683&0.034454664183&1.0512\\
\hline
8&0.013491456808&0.8788&0.012915023887&0.6472&0.019032557741&0.8562\\
\hline
16&0.006930734029&0.9610&0.007288727815&0.8253&0.009953295271&0.9352 \\
\hline
32&0.003529211695&0.9737&0.003870945517&0.9130&0.005085050476&0.9689\\
\hline
64&0.001839283189&0.9402&0.001994283984&0.9568&0.002569466386&0.9848 \\
\hline
128&0.001046515255&0.8136&0.001012041776&0.9786&0.001291388317&0.9925 \\
\hline
256&0.000720808717&0.5379&0.000510064738&0.9885&0.000647313735&0.9964 \\
\hline
 \end{tabular}
\end{table}

\section*{Acknowledgments}
The first author would like to thank the support of Alexander von Humboldt Foundation during the stay in University of Hamburg, Germany where this work was initialized. This work was supported by the National Basic Research Program of China under grant 2012CB821204, the National Nature Science Foundation of China under grant 11201464 and 91330115, and the scientific Research Foundation for the Returned Overseas Chinese Scholars, State Education Ministry. The second author gratefully acknowledges the support of the DFG Priority Program 1253 entitled ``Optimization with Partial Differential Equations".

\bibliographystyle{abbrv}


\nocite{Arada.N;Raymond.JP2002}
\nocite{Belgacem.FB;Bernardi.C;Fekih.HE2011}
\nocite{Arada.N;Raymond.JP2002}

\nocite{Fursikov.AV;Gunzburger.MD;Hou.LS1998}
\nocite{Grisvard.P1992}
\nocite{Hinze.M;Kunisch.K2004}
\nocite{Kunisch.K;Vexler.B2007}
\nocite{Thomee.V2006}
\nocite{Scott.LR;Zhang.SY1990}
\nocite{Gunzburger.MD;Hou.LS1992}

 \medskip
\end{document}